\documentclass[11pt, oneside]{article}
\usepackage{amsmath, amsthm, amssymb}
\usepackage{stmaryrd}
\usepackage{geometry}
\usepackage{natbib}
\usepackage{hyperref}
\usepackage[sort, nameinlink]{cleveref}
\usepackage{interval}
\hypersetup{colorlinks=true, citecolor=blue}

\usepackage{tikz}

\newcommand{\demph}[1]{\emph{#1}}
\newcommand{\G}{\Gamma}
\newcommand{\D}{\Delta}
\newcommand{\abst}[2]{\ensuremath{\{#1 \;|\; #2\}}}
\newcommand{\setset}{\textsc{Set-Set}}
\newcommand{\setfmla}{\textsc{Set-Fmla}}
\newcommand{\fmlaset}{\textsc{Fmla-Set}}
\newcommand{\fmlafmla}{\textsc{Fmla-Fmla}}

\newcommand{\lang}{\ensuremath{\mathcal{L}}}
\newcommand{\tuple}[1]{\ensuremath{\langle #1 \rangle}}
\newcommand{\dn}[1]{\ensuremath{\llbracket #1 \rrbracket}}
\newcommand{\sq}[2]{#1 \,\Yright\, #2}

\newcommand{\hook}{\ensuremath{\mathrel{\supset}}}

\newcommand{\mirror}[1]{\overline{#1}}
\newcommand{\malpha}{\mirror{\alpha}}
\newcommand{\frech}{Fr\'{e}chet-Hoeffding}
\newcommand{\modl}{\mathfrak{M}}
\newcommand{\alg}{\mathcal{A}}
\newcommand{\prob}{\ensuremath{\textsc{p}}}
\newcommand{\hopen}[2]{\ensuremath{\interval[soft open fences, open left]{#1}{#2}}}

\newcommand{\clos}[2]{\ensuremath{\interval{#1}{#2}}}

\newtheorem{lemma}{Lemma}
\newtheorem{fact}[lemma]{Fact}
\newtheorem{theorem}[lemma]{Theorem}
\newtheorem{corollary}[lemma]{Corollary}
\newtheorem{conjecture}[lemma]{Conjecture}
\theoremstyle{definition}
\newtheorem{defn}[lemma]{Definition}

\Crefname{fact}{Fact}{Facts}
\Crefname{defn}{Definition}{Definitions}
\Crefname{conjecture}{Conjecture}{Conjecture}

\newcommand{\dual}[1]{\ensuremath{#1^{\star}}}
\newcommand{\dalpha}{\dual{\alpha}}
\newcommand{\ddalpha}{\alpha^{\star\star}}
\newcommand{\ms}[2]{\ensuremath{\text{ms}^{#1}_{#2}}}

\newcommand{\tg}[3]{\ensuremath{#1^{#2}_{#3}}}

\newcommand{\pe}[1]{\textcolor{black}{#1}}

\title{Probabilistic consequence relations}
\author{Paul \'{E}gr\'{e}\thanks{IRL Crossing, CNRS. paul.egre@ens.fr}\ \and Ellie Ripley\thanks{Monash University. ripley@negation.rocks}}
\date{}%IRL Crossing, CNRS / Monash University}
\begin{document}
\maketitle

\sloppy

\begin{abstract}

\noindent \pe{This paper investigates logical consequence defined in terms of probability distributions, for a classical propositional language using a standard notion of probability. We examine three distinct probabilistic consequence notions, which we call material consequence, preservation consequence, and symmetric consequence. %While the first is fully classical, the second is structurally classical but operationally subclassical, and the third is structurally and operationally subclassical.
While material consequence is fully classical for any threshold, preservation consequence and symmetric consequence are subclassical, with only symmetric consequence gradually approaching classical logic at the limit threshold equal to 1. Our results extend earlier results obtained by J. Paris in a $\setfmla$ setting to the \setset\ setting, and consider open thresholds beside closed ones. In the \setset\ setting, in particular, they reveal that probability 1 preservation does not yield classical logic, but supervaluationism, and conversely positive probability preservation yields subvaluationism.}%In the \setfmla case, a classically valid argument can be defined indifferently in terms of truth preservation and in terms of probability 1 preservation. 

\end{abstract}

\section{Introduction}
\label{sec:introduction}

Logical validity for an argument is defined differently depending on whether one thinks of deductive or inductive arguments. In the deductive case, the standard definition of logical consequence is in terms of truth preservation (see \citealt{tarski1936concept,ladd1902validity}). In the inductive case, it is in terms of probability preservation (\citealt{skyrms1966choice,ladd1902validity}). Truth and probability are distinct notions, and one should not expect the preservation of the one to coincide with the preservation of the other, except in limit cases.

When arguments have a set of premises and a single conclusion (what we, following \citealt{humberstone:c}, call the \setfmla\ framework), classical logic is such a limit case.
In this framework, classical validity can be semantically characterized either in terms of preserving truth or in terms of preserving certainty (probability 1) (\citealt{hailperin1984probability,adams1998primer,paris2004deriving}). When the premises are not certain, however, then we get well-known departures from classical logic. The most famous example is given by lottery cases (\citealt{kyburg1997rule}): the probability of $A$ (``ticket 1 will lose'') and the probability of $B$ (``ticket 2 will lose'') can both exceed a threshold less than 1 (say $\frac{2}{3}$ in a 3-ticket lottery), but the probability of $A\wedge B$ can fall below that threshold (to as low as $\frac{1}{3}$).

Cases in which premises are believed with less than certainty are extremely common in everyday reasoning, and they raise a natural question: what logics govern the preservation of high (but possibly not 1) probability?
This question has received attention from various scholars, notably from Ernest Adams, Kevin Knight, and Jeff Paris (see \citealt{adams1996four,knight2003probabilistic,paris2004deriving}). Paris, in particular, produces a sound and complete axiomatization for logics defined in terms of probability preservation at or above a rational number between 0 and 1.

In this paper we are interested in a generalization of this question along two main directions. The first direction concerns the format of arguments. For Paris, and similarly for Adams and Knight, arguments are given in a multiple-premises single-conclusion setting (\setfmla). We are interested in the multiple-premises multiple-conclusions case (\setset) as well. One motivation for that is that this generalization casts a novel light on the characterization of some non-classical logics that coincide with classical logic in the \setfmla\ case, but that differ from it in the \setset\ case, namely supervaluationist logic on the one hand (\citealt{fine1975vagueness}), and subvaluationist logic on the other (\citealt{jaskowski1969propositional,varzi1994universal, hyde1997heaps}).

%The second direction concerns the definition of validity. Paris and Knight define validity in terms of surpassing or being equal to a threshold, namely in terms of belonging to a closed interval between $\alpha$ and 1. We are also interested in the case in which this interval might be open.

The second direction has to do with a variation on the notion of probability preservation, using a more general template entertained by Knight and Paris, allowing probabilistic thresholds to vary depending on premises and conclusions (see \citealt{paris2004deriving,knight2003probabilistic}). Consider a lottery case again: although the probability of a conjunction is typically less than the probability of either conjunct, there are well-known lower and upper bounds on the probability that a conjunction can take, namely the Fr\'echet-Hoeffding bounds (\citealt{frechet1935generalisation}), whereby $\max(0, p(A)+p(B)-1)\leq p(A\wedge B) \leq \min(p(A), p(B))$. Thus, when the probability of $A$ and that of $B$ are both at least $\frac{2}{3}$, the probability of the conjunction cannot be less than $\frac{1}{3}$. This observation is also central in the study of probabilistic coherence (see \citealt{knight2002measuring,biazzo2005probabilistic}), since it puts constraints on the attitude of rational belief that an agent ought to have in uncertain cases.

We use it to define a notion of validity which we call ``symmetric validity''. The notion basically requires that when premises are believed above some threshold $\alpha$, then not all conclusions can be disbelieved with probability less than $1-\alpha$. Unlike preservation logics, symmetric logics give us more texture regarding the relation between thresholds and logical operations. Unlike preservation logics, moreover, the resulting logics display interesting substructural, non-Tarskian features. Such logics have a closely related counterpart in the area of fuzzy logic (see \citealt{cobreros2024tolerance}), but they display fundamentally different properties.

Two main caveats need to be made before we proceed. The first is that, unlike in particular \cite{adams1996four}, in this paper we will not deal with the incorporation of a special conditional connective in the language to express the notion of conditional probability. The language will be perfectly classical; the only conditional we consider is the classical material conditional. %, and where conditional probability should be needed, it will be as an ancillary notion.
The second caveat is that the notion of probability we will be using is entirely classical too, unlike in recent work concerned with the incorporation of probability to non-classical logics (see \citealt{klein2021probabilities,egre2024certain}).

The material of this paper is organized as follows. First, in  \Cref{sec:probabilistic-models} we introduce the kinds of probabilistic models we use to define logical consequence. In \Cref{sec:material-consequence} we then start with a notion of consequence that allows us to build an exact probabilistic match with classical consequence in the \setset\ setting, and which we call material consequence. The next three sections deal with the study of preservation consequence: \Cref{sec:pres-cons} opens up with the special case of the extreme thresholds ${1}$ and $(0,1]$ to give a probabilistic characterization of super- and sub-valuationism. \Cref{sec:pres-general} then looks at the general case. \Cref{sec:pres-properties} establishes how many different preservation properties there are and discusses their structural properties. %Section \ref{sec:pres-cons} then investigates preservation consequence.
\Cref{sec:symm-cons} deals with symmetric consequence. Finally, \Cref{sec:conclusion} concludes on the way in which the three approaches relate to each other and to classical logic.

%Motivations: sorites, lottery?
%Mention Adams's (or someone else's: Hailperin?) preservation results? Paris's?
%Move to \setset, which is new.

\section{Probabilistic models}
\label{sec:probabilistic-models}

%Probability spaces, a la Moss

\subsection{Language and models}

Throughout this paper we work with a propositional language $\lang$ with a countable infinity of atomic sentences $p, q, r, \ldots$, a unary connective $\neg$ for negation and a binary connective $\lor$ for disjunction.
We also make use of connectives $\land, \bot, \top, \hook$; these are officially understood as defined from $\neg, \vee$ as usual.

\begin{defn}%[Arguments]
An argument $\sq{\G}{\D}$ is a pair of finite sets $\G, \D$ of sentences from $\lang$.
\end{defn}

To evaluate arguments, we need models, and we use what we call probabilistic models.

%This notion will serve two roles. The non-probabilistic part will allow us to state the standard definition of classical validity and germane notions of super- and sub-valuationist validity in terms of truth preservation. The probabilistic part will be needed to define the various notions of probabilistic consequence examined in this paper. The following definition is adapted from \cite{moss2018probabilistic}:

%\begin{defn}
%A classical model is a pair $M=\tuple{W,V}$ where $W$ is a nonempty set of worlds, and $V$ is a function from atomic sentences and worlds to $\{0,1\}$ such that $V(\neg A,w)=1-V(A,w)$, $V(A\wedge B,w)=\min(V(A,w),V(B,w))$ and $V(A\vee B,w)=\max(V(A,w),V(B,w))$. Given a model $M$, we write $\dn{A}=\{w\in W| V(w,A)=1\}$.
%\end{defn}

\begin{defn}%[Models] 
A probabilistic model is a quadruple $\modl = \tuple{W, \alg, \dn{\;}, \prob}$, where:
\begin{itemize}
  \item $W$ is a nonempty set of worlds;
  \item $\alg$ is an algebra on $W$, which is to say:
        \begin{itemize}
          \item $\alg \subseteq \wp(W)$,
          \item $\emptyset \in \alg$,
          \item for any $A \in \alg$, we have $W \setminus A \in \alg$, and
          \item for any $A, B \in \alg$, we have $A \cup B \in \alg$;
        \end{itemize}
  \item $\dn{\;}$ is a classical denotation function, a function $\lang \to \alg$ such that:
        \begin{itemize}
          \item $\dn{\neg \phi} = W \setminus \dn{\phi}$
          \item $\dn{\phi \lor \psi} = \dn{\phi} \cup \dn{\psi}$; and
        \end{itemize}
  \item $\prob$ is a probability function, a function $\alg \to \clos{0}{1}$ such that:
        \begin{itemize}
          \item $\prob(\emptyset) = 0$,
          \item $\prob(W \setminus A) = 1 - \prob(A)$, and
          \item if $A, B \in \alg$ and $A \cap B = \emptyset$, then $\prob(A \cup B) = \prob(A) + \prob(B)$.
        \end{itemize}
\end{itemize}
\end{defn}

Some notational shorthand involving these models will prove useful:

First, given such a model, we often treat $\prob$ also as a function $\lang \to [0, 1]$, writing just `$\prob(\phi)$' to mean $\prob(\dn{\phi})$.
In effect, we treat `$\prob$' as ambiguous between $\prob$ itself and $\prob \circ \dn{\;}$, trusting in context and in your indulgence.
Second, given such a model, we also say that `$w$ makes $\phi$ true', or that `$\phi$ is true at $w$', when $w \in \dn{\phi}$.
Third, we will sometimes write `$\modl_{\prob}$' to indicate some $\tuple{W, \alg, \dn{\;}, \prob}$, where we will have no need of further reference to $W, \alg,$ or $\dn{\;}$.

Note that in the case where $W$ is finite and where $\{w\} \in \alg$ for every $w \in W$, then to fully specify $\prob$ it suffices to give $\prob{\{w\}}$ for each $w \in W$, such that these probabilities of singletons sum to 1.
In this case we must have $\alg = \wp(W)$, and the probabilities of all other elements of $\alg$ are determined by finite additivity.\footnote{The difference between finite additivity and countable additivity doesn't matter for our purposes in this paper.
Finite additivity is enough for all our proofs to go through, so we don't require more; but in fact all the particular models we specify are finite, and so their finite additivity suffices for them to be countably additive as well.
Requiring countable additivity instead, then, wouldn't have any effect on what follows here.}

These models, with their modal structure, are handy for a number of the manipulations to follow, and they connect nicely to the structures used for example in \cite{moss2018probabilistic}.
Much of the literature on probabilistic logic, however, uses less complex structures, just probability distributions:

\begin{defn}
A probability distribution is a function $\lang \to [0, 1]$ such that:
\begin{itemize}
  \item $\prob(\bot) = 0$,
  \item $\prob(\neg \phi) = 1 - \prob(\phi)$, and
  \item if $\phi, \psi \vdash_{CL} \bot$ then $\prob(\phi \lor \psi) = \prob(\phi) + \prob(\psi)$, where $\vdash_{CL}$ is classical consequence.
\end{itemize}
\end{defn}

It is quick to see that, given any probabilistic model $\tuple{W, \alg, \dn{\;}, \prob}$, indeed $\prob \circ \dn{\;}$ (which, again, we often write just as `$\prob$') is always a probability distribution.
We can also go in the other direction, filling in any probability distribution to an entire probabilistic model.
However, it is more convenient for our purposes to show the following related fact: that we can take any probability distribution together with any \emph{finite} $\G \subseteq \lang$, and create a \emph{finite} probabilistic model that agrees with $\prob$ in its assignments of probabilities to every sentence in $\G$.

\begin{fact} \label{fact:power-set-model}
  Given any finite set $\G$ of sentences, let $At(\G)$ be the (necessarily finite) set of atomic sentences occurring in $\G$.
  Then for any probability distribution $\prob$, %\textcolor{blue}{based on a model $\modl_{\prob}$,} 
  there is a probabilistic model $\tg{\modl}{\G}{\prob} = \tuple{\tg{W}{\G}{\prob}, \tg{\alg}{\G}{\prob}, \tg{\dn{\;}}{\G}{\prob}, \tg{\prob}{\G}{\prob})}$ such that:
  \begin{itemize}
    \item $\tg{W}{\G}{\prob} = \wp(At(\G))$;
    \item $\tg{\alg}{\G}{\prob} = \wp(\tg{W}{\G}{\prob})$;
    \item for every atom $p \in At(\G)$, $\tg{\dn{p}}{\G}{\prob} = \abst{w \in \tg{W}{\G}{\prob}}{p \in w}$; and
    \item for every sentence $\gamma \in \G$, $\tg{\prob}{\G}{\prob}(\gamma) = \prob(\gamma)$.
  \end{itemize}
\end{fact}

\begin{proof}
  There is no decision to be made for $\tg{W}{\G}{\prob}$ and $\tg{\alg}{\G}{\prob}$; these are specified in the claim.
  So we just need to specify $\tg{\dn{\;}}{\G}{\prob}$ and $\tg{\prob}{\G}{\prob}$ in a way that meets the claims.
  We know $\tg{\dn{p}}{\G}{\prob}$ for all $p \in At(\G)$; for any $q \not\in At(\G)$, let $\tg{\dn{q}}{\G}{\prob} = \emptyset$.\footnote{This is just for concreteness; really it doesn't matter what $\tg{\dn{q}}{\G}{\prob}$ is when $q \not\in At(\G)$.}

  If $At(\G) = \{p_{1}, \ldots, p_{n}\}$, let a \demph{state description} be any sentence of the form $\pm p_{1} \land \ldots \land \pm p_{n}$, where $\pm p_{i}$ is either $p_{i}$ or $\neg p_{i}$.
  We have a bijection from state descriptions to worlds in $\tg{W}{\G}{\prob}$ given by $\tg{\dn{\;}}{\G}{\prob}$; let its inverse be $S$.
  (For example, if $n = 4$, then $S(\{p_{1}, p_{3}\}) = p_{1} \land \neg p_{2} \land p_{3} \land \neg p_{4}$.)

  To specify $\tg{\prob}{\G}{\prob}$ in full, it suffices to specify it on singletons.
  For each $w \in \tg{W}{\G}{\prob}$, let $\tg{\prob}{\G}{\prob}(\{w\}) = \prob(S(w))$.
  Now, since $\tg{\dn{S(w)}}{\G}{\prob} = \{w\}$, this gives us that for every $w$, we have $\tg{\prob}{\G}{\prob}(S(w)) = \prob(S(w))$.
  Moreover, every state description is $S(w)$ for some $w \in \tg{W}{\G}{\prob}$, so for every state description $s$ we have $\tg{\prob}{\G}{\prob}(s) = \prob(s)$.
  Since any two distinct state descriptions are classically inconsistent with each other, this in turn ensures that for any disjunction $\delta$ of state descriptions, $\tg{\prob}{\G}{\prob}(\delta) = \prob(\delta)$.
  But every $\gamma \in \G$ is classically equivalent to some disjunction of state descriptions, so for all such $\gamma$ we have $\tg{\prob}{\G}{\prob}(\gamma) = \prob(\gamma)$, as needed.\footnote{Taking $\bot$ to be the disjunction of the empty set of state descriptions.
    Similarly, the case where $At(\G) = \emptyset$ is perhaps not very interesting, but it is included in this reasoning, taking the conjunction of 0 conjuncts (the only state description in such a case) to be $\top$.}
\end{proof}

\subsection{Classical, super- and subvaluationist validity}

With these models in hand, we proceed to define a few familiar non-probabilistic notions of validity.
(These three notions make no use of the final coordinate in these models, but we stick to full probabilistic models for the sake of uniformity.)
Our first concept of validity is classical validity, which we define in terms of truth preservation:
\begin{defn}

We say that $\sq{\G}{\D}$ is \emph{classically} valid, written $\G \models_{CL} \D$, iff for every probabilistic model $\modl_{\prob}=\tuple{W, \alg, \dn{\;}, \prob}$ and every world $w\in W$, if $w$ makes every $\gamma\in \G$ true, then $w$ makes some $\delta\in \D$ true. Equivalently, $\sq{\G}{\D}$ is classically valid iff for every model $\modl_{\prob}$, $\dn{\bigwedge \G \hook \bigvee \D}=W$.

%$\dn{\bigwedge \G \hook \bigvee \D}=W$.

\end{defn}

We also consider two germane notions of logical consequence: supervaluationist validity is preservation of super-truth (truth at every world), and subvaluationist validity is preservation of sub-truth (truth at some world).

\begin{defn}

We say that $\sq{\G}{\D}$ is \emph{supervaluationistically} valid, written $\G \models_{SV} \D$, iff for every probabilistic model $\modl_{\prob}=\tuple{W, \alg, \dn{\;}, \prob}$, if every $\gamma$ is true at every world of $W$, then some $\delta$ is true at every world of $W$. Equivalently, $\sq{\G}{\D}$ is supervaluationistically valid iff, if $\dn{\gamma}=W$ for every $\gamma \in G$, then $\dn{\delta}=W$ for some $\delta\in G$.

\end{defn}

\begin{defn}

We say that $\sq{\G}{\D}$ is \emph{subvaluationistically} valid, written $\G \models_{sV} \D$, iff for every probabilistic model $\modl_{\prob}=\tuple{W, \alg, \dn{\;}, \prob}$, if every $\gamma\in \G$ is true at some world $w$, then some $\delta\in \D$ is true at some world; Equivalently, $\sq{\G}{\D}$ is \emph{subvaluationistically} valid iff, if $\dn{\gamma}\neq \emptyset$ for every $\gamma \in \G$, then $\dn{\delta}\neq \emptyset$ for some $\delta\in \D$.

\end{defn}

While supervaluationist validity and subvaluationist validity coincide with classical validity in the $\setfmla$ case and the $\fmlaset$ case, respectively, they differ in the $\setset$ framework (see \citealt{hyde1997heaps,ripley2013sorting}). In particular, $p\vee \neg p \not\models_{SV} p, \neg p$, and $p, \neg p\not\models_{sV} p\wedge \neg p$, that is we lose abjunction and adjunction, respectively, in those frameworks.

%Define classical validity;
%Define super/sub validity here as well; all over probabilistic models.

%Mention LP/K3, and why they won't appear here.
%Reinforce that our probabilities are completely classical; mention some nonclassical probability stuff (R.Williams, Paul \& Lorenzo \& Jan, others?, Klein et al).

\subsection{Probabilistic consequence: The general recipe}
\label{sec:general-recipe}

We defined classical validity, supervaluationist validity, and subvaluationist validity in terms of truth preservation, but we shall see that each has an equivalent characterization in terms of probability preservation, which is what drives the introduction of probabilistic models in the first place.

Our notions of consequence over these probabilistic models all follow a general recipe.
First, we take for granted some set $\alpha \subseteq \clos{0}{1}$ that is to count as the `good' probabilities: sentences with probabilities in $\alpha$ are those with a probability high enough.
We remain neutral throughout as to what such probabilities are high enough \emph{for}; we hope our results here can be useful to a range of possible interpretations and applications.

We make three assumptions about our set $\alpha$ of `good' probabilities: it contains 1, it does not contain 0, and it is an \demph{upset}, in the sense that for any $x, y \in \clos{0}{1}$ with $x \leq y$, if $x \in \alpha$ then $y \in \alpha$.
We refer to all such sets simply as \demph{upsets}, taking for granted the conditions about 1 and 0.
To fix ideas and notation, note that every upset $\alpha$ has an infimum $\inf \alpha$, which we call $\alpha$'s \demph{threshold}, and that for any $0 < x < 1$, there are exactly two upsets with threshold $x$: namely, $\hopen{x}{1}$ and $\clos{x}{1}$.\footnote{Since all upsets exclude 0, there is only an open upset with threshold 0; and since all upsets include 1, there is only a closed upset with threshold 1.}
We call upsets $\hopen{x}{1}$ \demph{open} and upsets $\clos{x}{1}$ \demph{closed}.
This gives a convenient way to specify any upset: simply by giving its threshold and saying whether it is open or closed.

Let a \demph{counterexample notion} be a three-place relation between upsets, probabilistic models, and arguments.
Given a counterexample notion and an upset $\alpha$, we always determine a consequence relation following the same recipe: count an argument as valid iff there is no probabilistic model that bears the counterexample notion relation to $\alpha$ and that argument. 

%%%

%\input{alpha-sat}

\section{Material consequence}
\label{sec:material-consequence}

The bulk of the paper considers two main counterexample notions in the probabilistic setting, which we will call \demph{preservation consequence} and \demph{symmetric consequence}.
As a warm-up and an aid to later discussion, in this brief section we first explore a distinct option, which we call \demph{material consequence}.

\begin{defn}
  Given an upset $\alpha$, a probabilistic model $\modl_{\prob}$ is an \demph{$\alpha$ material counterexample} to an argument $\sq{\G}{\D}$ iff $\prob(\bigwedge \G \hook \bigvee \D) \not\in \alpha$.
  Thus, the argument $\sq{\G}{\D}$ is $\alpha$-materially valid iff every $\modl_{\prob}$ is such that $\prob(\bigwedge \G \hook \bigvee \D) \in \alpha$.
\end{defn}

That is, to see whether a probabilistic model is a material counterexample to an argument, we first roll the entire argument up into a single sentence, and then check the probability of that sentence on the model.
If the probability is not high enough (is not in the specified upset), then we have a counterexample; if the probability is high enough (is in the upset), then we do not have a counterexample.
The arguments that are $\alpha$-materially valid, then, are those whose associated sentences always have probabilities in $\alpha$.
This, we think, is a reasonable enough notion.
We know how to assign probabilities to sentences, and there is a natural way of associating a sentence to each argument; material consequence results from putting these two ideas together.

One upshot of this idea, though, is that the dependence on an upset $\alpha$ is an illusion:
\begin{fact}\label{fact:material}
  For any upset $\alpha$, the argument $\sq{\G}{\D}$ is $\alpha$-materially valid iff it is classically valid.
\end{fact}

\begin{proof}
  From the definition of material validity, $\sq{\G}{\D}$ is $\alpha$-materially valid iff for all probability models $\modl_{\prob}$, we have $\prob(\bigwedge \G \hook \bigvee \D) \in \alpha$.
  And it's well-known that $\sq{\G}{\D}$ is classically valid iff $\bigwedge \G \hook \bigvee \D$ is a classical tautology.
  So it's enough to show that for any upset $\alpha$ and any sentence $\phi$, there is a probability model $\modl_{\prob}$ with $\prob(\phi) \not\in \alpha$ iff $\phi$ is not a classical tautology.

  All classical tautologies $\phi$ have probability 1 on every probability model since $\dn{\phi}=W$, so the left-to-right direction is immediate by contraposition.
  For the right-to-left direction, note that for any $\phi$ that is not a classical tautology, \pe{there is a probability model $\modl_{\prob}$ such that $\dn{\phi}\neq W$.
  Let $\prob'$ be a probability function based on the same model such that $\prob'(\dn{\phi})=0$ and $\prob'(W \backslash \dn{\phi})=1$; such a probability function exists, as it suffices that it concentrates all the mass on a single world in $W \backslash \dn{\phi}$. By \Cref{fact:power-set-model}, there is a model $\modl_{\prob'}^{\phi}$ in which $\tg{\prob'}{\phi}{\prob'}(\phi) = \prob'(\phi)=0$.
  }
  %where $\prob(A) = 0$. \textcolor{blue}{}
\end{proof}

Material validity, then, always perfectly matches classical validity, regardless of which upset we choose.
This fact is interesting in its own right, and it also helps circumscribe the applications where material validity might be of some interest: those where reasoning under uncertainty should nonetheless hold exactly to classical standards.

In the rest of the paper, however, we explore other options besides this.
We are particularly interested in approaches to reasoning under uncertainty that capture a different kind of insight: the idea that classicality should emerge at the limit, when things are certain (that is, at the upset $\{1\}$); but that some nonclassical features can be appropriate when more uncertainty is in the air (that is, for looser upsets).

\section{Preservation consequence: the $\{1\}$ and $(0,1]$ cases}
\label{sec:pres-cons}

In this section we introduce the probabilistic counterexample notion that is arguably the most natural, namely \demph{preservation}. We start with some examples and then characterize the two extremes of preservation consequence, which turn out to match super- and sub-valuationist validity.

\subsection{Definition}
\label{sec:pc-definition}

%Perhaps the most natural counterexample notion is the one we call \demph{preservation}:

\begin{defn} \label{defn:pres-con}
A probabilistic model $\modl_{\prob}$ is an $\alpha$-preservation counterexample to an argument $\sq{\G}{\D}$ iff $\prob[\G] \subseteq \alpha$ and $\prob[\D] \subseteq \clos{0}{1} \setminus \alpha$.
Thus, the argument $\sq{\G}{\D}$ is $\alpha$-preservation valid iff every $\modl_{\prob}$ is such that if $\prob[\G] \subseteq \alpha$, then there is some $\delta \in \D$ with $\prob(\delta) \in \alpha$.
\end{defn}

This is just like usual designated-values approaches to defining consequence, here using probabilities in $\alpha$ as our designated values.
An $\alpha$-preservation counterexample to an argument is one that takes all the premises of the argument, and none of its conclusions, to designated values.
An argument is $\alpha$-preservation valid, then, where there is no way to do this, when any probabilistic model that gives all premises probabilities in $\alpha$ must also give some conclusion a probability in $\alpha$.

For example, consider the upset $\hopen{.7}{1}$ and the probabilities associated with a single roll of a fair 6-sided die.
Let $p$ be the proposition that the die comes up $> 1$, let $q$ be the proposition that the die comes up $< 6$, and let $\prob$ come from a probabilistic model that assigns appropriate probabilities to this situation.
Then $\prob(p) = \prob(q) = \frac{5}{6} > .7$, while $\prob(p \land q) = \frac{4}{6} < .7$.
So the argument $\sq{p, q}{p \land q}$ is not $\hopen{.7}{1}$-preservation valid; there is a $\hopen{.7}{1}$-preservation counterexample.
(As we will see, the argument $\sq{p, q}{p \land q}$ is $\alpha$-preservation valid only for one choice of $\alpha$, namely, $\{1\}$.)

On the other hand, consider the upset $\clos{.4}{1}$ and the argument $\sq{p \land \neg q, q \land \neg p, \neg(p \lor q)}{\bot}$.
Note that any $\clos{.4}{1}$-preservation counterexample to this argument would have to assign a probability $\geq .4$ to each of $p \land \neg q$, $q \land \neg p$, and $\neg (p \lor q)$.
Since these three sentences are pairwise incompatible, this isn't possible, since the sum of those values would have to exceed 1.
So there can be no $\clos{.4}{1}$-preservation counterexample to this argument, and the argument is $\clos{.4}{1}$-preservation valid.
On the other hand, there are probabilistic models that assign a probability of $\frac{1}{3}$ to each of these sentences, so the same argument is not $\clos{.3}{1}$-preservation valid.%\footnote{Question, probably not to be included in the final paper: is there any argument $\sq{\G}{\D}$ with both $\G$ and $\D$ nonempty and upset $\alpha \neq \{1\}$ such that: 1) $\sq{\G}{\D}$ is $\alpha$-preservation valid; 2) there is no $\gamma \in \G$ such that $\sq{\gamma}{\D}$ is $\alpha$-preservation valid; 3) there is no $\delta \in \D$ such that $\sq{\G}{\delta}$ is $\alpha$-preservation valid?
%If so, is there any such with $\clos{.5}{1} \subseteq \alpha$?}

\subsection{Super and subvaluationism}
\label{sec:super-subv}

Given our setup, there are two extreme upsets: the smallest upset $\{1\}$ and the largest upset $\hopen{0}{1}$.\footnote{Recall that we require that every upset includes $1$ and excludes $0$.
We leave consideration of the situation involving $\{\}$ and $\clos{0}{1}$ as exercises for the interested reader.}
In this section, we show that the preservation consequence relations associated with these extreme upsets are familiar from non-probabilistic work.

\begin{fact} \label{fact:1super}
  An argument $\sq{\G}{\D}$ is $\{1\}$-preservation valid iff it is supervaluationistically valid.
\end{fact}

\begin{proof}
Left to right: suppose that there is a supervaluational counterexample to $\sq{\G}{\D}$.
This is a probabilistic model $\tuple{W, \alg, \dn{\;}, \prob}$ such that for every $\gamma \in \G$ we have $\dn{\gamma} = W$ and for each $\delta \in \D$ we have $\dn{\delta} \neq W$.
We're going to use this to generate a new probabilistic model $\tuple{W, \alg, \dn{\;}, \prob'}$.

Since $\D$ is finite, let $n$ be the number of its members; and for every $\delta_{i} \in \D$, choose some world $w_{i} \not\in \dn{\delta_{i}}$.
Now let $\prob'$ be the probability distribution that assigns probability $\frac{1}{n}$ to each such $w_{i}$ and probability 0 to all other worlds.

To see that $\tuple{W, \alg, \dn{\;}, \prob}$ is a $\{1\}$-preservation counterexample to $\sq{\G}{\D}$, note first that $\dn{\gamma} = W$ for each $\gamma \in \G$, so $\prob(\gamma) = 1$ for each of these as well.
And also note that since $w_{i} \not\in \dn{\delta_{i}}$ for each $\delta_{i} \in \D$, each such $\delta_{i}$ can have probability at most $\frac{n-1}{n}$, which is less than 1.
So we have our counterexample.\smallskip

Right to left: suppose that there is a $\{1\}$-preservation counterexample to $\sq{\G}{\D}$.
This is a probabilistic model $\tuple{W, \alg, \dn{\;}, \prob}$ such that $\prob(\gamma) = 1$ for every $\gamma \in \G$ and $\prob(\delta) \neq 1$ for every $\delta \in \D$.
We're going to use this to generate a new probabilistic model $\tuple{W', \alg', \dn{\;}', \prob'}$.

For $W', \alg', \dn{\;}'$, the idea is essentially just to throw out any worlds outside $\dn{\bigwedge{\G}}$.
That is, let $W' = \dn{\bigwedge \G}$, let $\alg' = \abst{A \cap \dn{\bigwedge \G}}{A \in \alg}$, and let $\dn{A}' = \dn{A} \cap \dn{\bigwedge\G}$.
This idea extends to $\prob'$ too.
Any $A' \in \alg'$ is $A \cap \dn{\bigwedge \G}$ for some $A \in \alg$, and so $A' \in \alg$ as well; we simply let $\prob'(A') = \prob(A')$. We know $\prob'$ is a probability distribution, since $\prob'(W')=1$ by hypothesis, and $\prob'(\emptyset)=0$, and finite additivity results from the fact that the $A'$ are elements of $\alg$.
This then gives us that $\prob'(\phi) = \prob(\phi)$ for any sentence $\phi$.

To see that $\tuple{W', \alg', \dn{\;}', \prob'}$ is a supervaluational counterexample to $\sq{\G}{\D}$, note first that $\dn{\gamma} = W'$ for each $\gamma \in \G$.
Then, note that, since $\prob(\delta) \neq 1$ for every $\delta \in \D$, we also have $\prob'(\delta) \neq 1$ for all such $\delta$.
Since $\prob'(W') = 1$, for each such $\delta$ there must be some $w \in W'$ with $w \not\in \dn{\delta}'$.
So we have our counterexample.
\end{proof}

The situation with the other extreme upset is, as you might expect, the mirror image:

\begin{fact} \label{fact:01sub}
  An argument $\sq{\G}{\D}$ is $\hopen{0}{1}$-preservation valid iff it is subvaluationistically valid.
\end{fact}

\begin{proof}
  As \Cref{fact:1super}, mutatis mutandis.
\end{proof}

Note that any \setset\ consequence relation determines a particular \setfmla\ consequence relation, but that many different \setset\ consequence relations can determine the same \setfmla\ consequence relation.
For example, the classical and supervaluational \setset\ consequence relations are distinct, but they share their \setfmla\ fragment.
Because of this, working in a \setset\ framework allows us to see differences that are invisible through a \setfmla\ lens.

Along these lines, \Cref{fact:1super} reveals more complexity behind the claim, made for example in \citet[p.\ 26]{adams1998primer}, that preservation of probability 1 and classical entailment coincide.
This claim, as made there, is true, since the context fixes that what's meant is the \setfmla\ fragments of these relations.
But \Cref{fact:1super} goes farther, showing that $\{1\}$-preservation consequence is not fully classical, when the extra texture visible in a \setset\ framework is considered.
Moreover, this extra texture reveals $\{1\}$-preservation consequence to be supervaluational consequence.

As the \fmlaset\ framework is less commonly studied, we don't know of anyone who has claimed that $\hopen{0}{1}$-preservation consequence is classical.
We can see, though, that there would be a kind of justice to such a claim: it is indeed classical in the \fmlaset\ framework, just not in the full \setset\ framework we work with in this paper.

The characterization of super- and sub-valuationism in terms of probability preservation sheds a specific light on a property often stressed of both frameworks (for example in \citealt{williamson1994vagueness}), which is that super-truth and sub-truth fail to be truth-functional.
That is, in the same way the probability of a disjunction cannot be determined just from the probabilities of its disjuncts, the super-truth or otherwise of a disjunction cannot be determined just from whether its disjuncts are super-true or not.
(For more detailed discussion of truth-functionality, see \citealt{chemla2019suszko} or \citealt[\S 3.1]{humberstone:c}.)
So while truth-preservation and $\alpha$-probability-preservation may appear to coincide in the $\setfmla$ case when $\alpha=1$,  the $\setset$ case reveals that the notions behave fundamentally differently, even in that extreme case.\footnote{The reader may also wonder about the connection between $\alpha$-preservation consequence and the analog in the case of \L ukasiewicz's fuzzy propositional logics in which the set of truth values is $[0,1]$ and validity is defined as the preservation of the degree $\alpha$ from premises to conclusions, restricting the connectives to negation, conjunction, and disjunction (see  \citealt{hajek1998metamathematics} and \citealt{bergmann2008introduction} for more information on these logics). Assuming $v(\neg A)=1-v(A)$, $v(A\wedge B)=min(v(A),v(B))$, and $v(A\vee B)=max(v(A),v(B))$, $q$ fails to entail $p\vee \neg p$ in \L ukasiewicz's logic when $\alpha=1$, unlike supervaluationism. And for thresholds $\alpha$ below 1, the resulting \L ukasiewicz logics will preserve adjunction, unlike the corresponding probabilistic $\alpha$-preservation consequence relations. So probabilistic $\alpha$-preservation logics and their fuzzy counterparts differ across the board in the \setfmla\  case.} 

%The answer depends in part on how the semantics for conjunction and disjunction is defined, assuming for negation that $v(\neg A)=1-v(A)$. When $v(A\vee B)=max(v(A),v(B))$, so called weak disjunction, the logic differs from supervaluationism, since $q$ fails to entail $p\vee \neg p$ (assign $q$ the value 1 and $p$ the value $1/2$). There is an alternative semantics in terms of strong disjunction, where $v(A\vee B) =max(1,v(A)+v(B))$. But then $p\vee p$ fails to entail $p$, which also departs from supervaluationism. These divergences between \L ukasiewicz's logic and preservation consequence logic hold at other thresholds $\alpha \geq .5$. More generally, the corresponding logics differ at all thresholds in the \setfmla\ case: \L ukasiewicz's logics with weak conjunction will preserve adjunction at all thresholds unlike their probabilistic counterpart. }

%Using the corresponding conjunctions, these divergences can be extended to arbitrary thresholds, and since all probabilistic preservation relations coincide with classical logic in the $\fmlafmla$ case (see \Cref{fact:full-narrowing-preserves-setfmla,fact:full-widening-preserves-fmlaset} further below), this suffices to show that \L ukasiewicz's fuzzy logics differ from probabilistic preservation logics at all thresholds.
%}

We pause to note some results about supervaluational and subvaluational consequence shown in \citet[pp.\ 237--238]{kremer:SupervaluationbasedConsequenceRelations2003}:
\begin{itemize}
  \item if $\sq{\G}{\D}$ is supervaluationistically valid, then either $\D$ is empty or there is some $\delta \in \D$ where $\sq{\G}{\delta}$ is supervaluationistically valid; and
  \item if $\sq{\G}{\D}$ is subvaluationistically valid, then either $\G$ is empty or there is some $\gamma \in \G$ where $\sq{\gamma}{\D}$ is subvaluationistically valid.
\end{itemize}
So we can immediately conclude the same of $\{1\}$-preservation and $\hopen{0}{1}$-preservation, respectively.
Indeed, we will now work our way up to some results---named \Cref{fact:setfmla-to-setset,fact:fmlaset-captured} below---that in some sense extend these results to intermediate choices of upset.\footnote{Unfortunately, we didn't find a way to adapt \cite{kremer:SupervaluationbasedConsequenceRelations2003}'s elegant proofs of these above claims to our more general setting, so we've had to take a different, less elegant, approach.}

\section{Preservation consequence: general case}\label{sec:pres-general}
%\subsection{Discussion of consequence relations}
%\label{sec:pc-order-cons-relat}

There is some interesting texture to explore in the preservation consequence relations, outside the two extreme upsets $\{1\}$ and $\hopen{0}{1}$ that turn out to determine familiar consequence relations.
We open our discussion of this texture by developing some ideas and background, on the way to showing, in \Cref{sec:sufficient-for-preservation-invalidity}, three sufficient conditions for $\alpha$-preservation invalidity.
These sufficient conditions give us a grip on $\alpha$-preservation consequence for upsets besides these extremes.

\subsection{$\alpha$-satisfiability, $\alpha$-tautology, and dual upsets}
\label{sec:alpha-satisf-alpha}

For any upset $\alpha$, we have the following notions:
\begin{defn} \label{defn:sat-taut}
  A set $\G$  is $\alpha$-satisfiable iff there is some probabilistic model $\tuple{W, \alg, \dn{\;}, \prob}$ such that $\prob(\gamma) \in \alpha$ for each $\gamma \in \G$.
  A set $\D$ is $\alpha$-tautologous iff for every probabilistic model $\tuple{W, \alg, \dn{\;}, \prob}$, there is some $\delta \in \D$ such that $\prob(\delta) \in \alpha$.
\end{defn}

Note that $\alpha$-unsatisfiability and $\alpha$-tautology connect to special cases of $\alpha$-preservation validity: $\G$ is $\alpha$-unsatisfiable iff $\sq{\G}{\emptyset}$ is $\alpha$-preservation valid; and $\D$ is $\alpha$-tautologous iff $\sq{\emptyset}{\D}$ is $\alpha$-preservation valid.

\begin{defn} \label{defn:mirror-dual}
 The \demph{mirror image} $\malpha$ of $\alpha$ is $\abst{x \in \clos{0}{1}}{1 - x \in \alpha}$; and the \demph{dual} $\dalpha$ of $\alpha$ is $\clos{0}{1} \setminus \malpha$.
\end{defn}

It is immediate that $\dalpha$ is an upset, and that $\ddalpha = \alpha$.
It's also quick to see that no upset $\alpha$ can be self-dual, as $\dalpha$ contains $.5$ iff $\alpha$ does not.
These notions are interrelated as follows:

\begin{fact} \label{fact:tautology-satisfiability-dual}
  For any upset $\alpha$ and argument $\sq{\G}{\D}$, the argument $\sq{\G}{\D}$ is $\alpha$-preservation valid iff $\sq{\neg \D}{\neg \G}$ is $\dalpha$-preservation valid.

  As a special case of this, $\alpha$ and set $\D$ of sentences, $\D$ is $\alpha$-tautologous iff $\neg\D$ is $\dalpha$-unsatisfiable.
\end{fact}
\begin{proof}
  Spelling out definitions, recalling that $\prob(\neg \delta) = 1 - \prob(\delta)$ for any probability distribution $\prob$ and sentence $\delta$.
\end{proof}

\Cref{fact:tautology-satisfiability-dual} will be useful in what follows particularly when we focus on \setfmla\ and \fmlaset\ fragments of $\alpha$-preservation consequence relations, as it allows us to turn results about one of these fragments into results about the other, just for the dual upset.

In this paper, we do not explore $\alpha$-satisfiability and $\alpha$-tautology in any depth.
By \Cref{fact:tautology-satisfiability-dual}, it would be enough to focus on just one of these, but we do not develop any detailed picture of either one.\footnote{For a valuable and in-depth discussion of $\alpha$-satisfiability, see \cite{knight2002measuring}---noting that Knight speaks of `$\eta$-consistency' to mean what we would call `$\clos{\eta}{1}$-satisfiability'.
(That paper does not consider open upsets.)
We will be drawing on results from this paper as we go.}
Instead, we use these notions to describe and explore our real target: $\alpha$-preservation.

\subsection{Background from \cite{adams:UncertaintiesTransmittedPremises1975}}
\label{sec:backgr-from-adams-levine}

The results in this subsection are taken from \cite{adams:UncertaintiesTransmittedPremises1975}, although we restate them here in the forms we'll need them in.
We also prove them here, as they are not explicitly proved in \cite{adams:UncertaintiesTransmittedPremises1975}.

\begin{defn}
  Given a classically-valid \setfmla\ argument $\sq{\G}{\phi}$, say that a set $\G' \subseteq \G$ is \demph{minimally sufficient} iff $\sq{\G'}{\phi}$ is classically valid and there is no proper subset $\G'' \subsetneq \G'$ such that $\sq{\G''}{\phi}$ is classically valid.

  For any classically-valid \setfmla\ argument $\sq{\G}{\phi}$ whose minimally sufficient sets are $\G'_{1}, \ldots, \G'_{n}$, let $\ms{\G}{\phi}$ be the sentence $\bigvee_{1 \leq i \leq n} (\bigwedge \G'_{i})$.
\end{defn}

\begin{fact} \label{fact:ms-equiconsistent}
  For any classically-valid \setfmla\ argument $\sq{\G}{\phi}$ and any $\Sigma \subseteq \G$, the set $\Sigma \cup \{\neg \phi\}$ is classically consistent iff the set $\Sigma \cup \{\neg \ms{\Gamma}{\phi}\}$ is classically consistent.
\end{fact}

\begin{proof}
  Recall that $\ms{\G}{\phi}$ is a disjunction each of whose disjuncts classically entails $\phi$; so $\sq{\ms{\G}{\phi}}{\phi}$ is classically valid.
  But if $\Sigma \cup \{\neg \ms{\G}{\phi}\}$ is inconsistent then $\sq{\Sigma}{\ms{\G}{\phi}}$ is classically valid; so by transitivity of classical validity $\sq{\Sigma}{\phi}$ would have to be classically valid, and thus $\Sigma \cup \{\neg \phi\}$ classically inconsistent.

  For the other direction, suppose that $\Sigma \cup \{\neg \phi\}$ is classically inconsistent; then $\sq{\Sigma}{\phi}$ is classically valid.
  Then, since $\Sigma \subseteq \G$, there must be some minimally sufficient $\G'$ such that $\G' \subseteq \Sigma$, and so $\sq{\Sigma}{\bigwedge \G'}$ is classically valid.
  But then $\sq{\Sigma}{\ms{\G}{\phi}}$ is classically valid as well, and so $\Sigma \cup \{\neg \ms{\G}{\phi}\}$ is classically inconsistent.
\end{proof}

\begin{fact}[see \protect{\citealt[p.\ 434]{adams:UncertaintiesTransmittedPremises1975}}] \label{fact:inconsistent-ms-to-zero}
  Suppose the \setfmla\ argument $\sq{\G}{\phi}$ is classically valid, and that all its minimally sufficient sets are classically inconsistent.
  Then for any probabilistic model $\modl_{\prob}$, there is a probabilistic model $\modl_{\prob'}$ such that:
  \begin{itemize}
    \item for all $\gamma \in \G$, we have $\prob'(\gamma) = \prob(\gamma)$, and
    \item $\prob'(\phi) = 0$.
  \end{itemize}
\end{fact}

\begin{proof}
  Let $\modl_{\prob}$ be given.
  We move first to $\modl^{\star} = \tg{\modl}{\G \cup \{\phi\}}{\prob}$, as given by \Cref{fact:power-set-model}.
  That is, $\modl^{\star} = \tuple{W^{\star}, \alg^{\star}, \dn{\;}^{\star}, \prob^{\star}}$, where:
  \begin{itemize}
    \item $W^{\star} = \wp(At(\G \cup \{\phi\}))$;
    \item $\alg^{\star} = \wp(W^{\star})$;
    \item for all $p \in At(\G \cup \{\phi\})$, we have $\dn{p}^{\star} = \abst{w \in W^{\star}}{p \in W}$; and
    \item for all $\gamma \in \G \cup \{\phi\}$, we have $\prob^{\star}(\gamma) = \prob(\gamma)$.
  \end{itemize}

  Now, we arrive at the desired $\modl_{\prob'}$ by modifying just the last coordinate of $\modl^{\star}$.
  That is, $\modl_{\prob'} = \tuple{W^{\star}, \alg^{\star}, \dn{\;}^{\star}, \prob'}$, and it remains just to specify $\prob'$ and to show the claimed results about it.

  First, note that since all of $\sq{\G}{\phi}$'s minimally sufficient sets are inconsistent, so too is their disjunction, which is to say that $\ms{\G}{\phi}$ is inconsistent.
  As such, we know that $\dn{\ms{\G}{\phi}}^{\star} = \emptyset$.

  Consider, then, any $w \in \dn{\phi}^{\star}$.
  Let $\Sigma_{w} = \abst{\gamma}{\gamma \in \G \text{ and } w \in \dn{\gamma}^{\star}} \cup \abst{\neg \gamma}{\gamma \in \G \text{ and } w \not\in \dn{\gamma}^{\star}}$.
  The set $\Sigma_{w} \cup \{\neg \ms{\G}{\phi}\}$ is classically consistent, since all its members are true at $w$, and so by \Cref{fact:ms-equiconsistent} the set $\Sigma_{w} \cup \{\neg \phi\}$ is also classically consistent.
  Since $At(\Sigma_{w} \cup \{\neg \phi\}) \subseteq At(\G \cup \{\phi\})$, as we've defined $W^{\star}$ and $\dn{\;}^{\star}$ this is enough to conclude that there is some $w^{\dagger} \in W^{\star}$ such that for every $\sigma \in \Sigma_{w} \cup \{\neg \phi\}$, we have $w^{\dagger} \in \dn{\sigma}^{\star}$.
  And for any $w \not \in \dn{\phi}^{\star}$, let $w^{\dagger} = w$.
  This gives us, for each $w \in W^{\star}$, some corresponding $w^{\dagger}$ where: 1) $w^{\dagger} \not\in \dn{\phi}^{\star}$, and 2) for every $\gamma \in \G$, we have $w \in \dn{\gamma}^{\star}$ iff $w^{\dagger} \in \dn{\gamma}^{\star}$.

  Now, for any $w \in W^{\star}$, let $^{\dagger}w$ be its $^{\dagger}$-preimage, the set $\abst{x}{x^{\dagger} = w}$.
  Note that whenever $w \in \dn{\phi}'$, then the set $^{\dagger}w$ is empty.
  Using this, we define $\prob'$ by defining it on singletons as $\prob'(\{w\}) = \sum_{x \in ^{\dagger}w} \prob^{\star}(x)$.

  In effect, we are moving from $\prob^{\star}$ to $\prob'$ by taking the weight of probability at each world $w$ and shifting it to $w^{\dagger}$.
  Any worlds $w \not\in \dn{\phi}^{\star}$ retain all the probability they begin with under this operation, since for all such worlds $w^{\dagger} = w$.
  Such worlds might, however, end up with more total probability, since such a $w$ might still be $x^{\dagger}$ for some $x \neq w$.
  And any worlds $w \in \dn{\phi}^{\star}$ end up with 0 probability, since for all such worlds there is no $x$ with $x^{\dagger} = w$.

  Clearly, then, we have $\prob'(\phi) = 0$, as desired.
  It remains to show that $\prob'(\gamma) = \prob^{\star}(\gamma)$ for all $\gamma \in \G$.
  Take any such $\gamma$, then, with $\dn{\gamma}^{\star} = \{w_{1}, \ldots, w_{n}\}$.
  We know $\prob^{\star}(\gamma) = \prob^{\star}(\{w_{1}\}) + \ldots + \prob^{\star}(\{w_{n}\})$, and that $\prob'(\gamma) = \prob'(\{w_{1}\}) + \ldots + \prob'(\{w_{n}\})$.
  But for any of these $w_{i}$, the world $w_{i}^{\dagger}$ must be one of these $w_{j}$, since we know that $w_{i} \in \dn{\gamma}^{\star}$ iff $w_{i}^{\dagger} \in \dn{\gamma}^{\star}$.
  And similarly, for any $w \in W^{\star}$ that is not among the $w_{i}$s, $w^{\dagger}$ is also not among them, for the same reason.
  So while we might shift some probability from addend to addend, we do not shift any probability away from or into the whole collection, and so the two sums must be the same.
  Thus, we have what we're after.
\end{proof}

\subsection{Sufficient conditions for preservation invalidity}
\label{sec:sufficient-for-preservation-invalidity}

Here, we use the foregoing to compile some results about conditions that are sufficient to show a given argument to be $\alpha$-preservation invalid.
First, we prove that all $\alpha$-preservation consequence relations are subclassical:

\begin{fact} \label{fact:preservation-is-subclassical}
  For any upset $\alpha$, if $\sq{\G}{\D}$ is classically invalid, then it is $\alpha$-preservation invalid.
\end{fact}

\begin{proof}
  Take a one-world probabilistic model $\modl = \tuple{\{w\}, \wp{\{w\}}, \dn{\;}, \prob}$ where $\dn{\;}$ is a classical valuation that provides a counterexample to $\sq{\G}{\D}$.
  In this model, $\prob(\gamma) = 1$ for every $\gamma \in \G$ and $\prob(\delta) = 0$ for every $\delta \in \D$, so this is an $\alpha$-preservation counterexample to $\sq{\G}{\D}$, regardless of $\alpha$.
\end{proof}

Before moving to our next result, we pause to cite a theorem due to Adams, as we draw on this result in what follows:

\begin{theorem}[\protect{\citealp[Thm.\ 14, p.\ 39]{adams1998primer}}] \label{fact:adams-thm14}
  Take a classically valid \setfmla\ argument $\sq{\G}{\delta}$, where $\G = \{\gamma_{1}, \ldots, \gamma_{n}\}$.
  If $\G$ is classically consistent, and if $\sq{\G'}{\delta}$ is classically invalid for every $\G' \subsetneq \G$, then for any sequence $x_{1}, \ldots, x_{n}$ of numbers from $\clos{0}{1}$ such that $\sum_{1 \leq i \leq n} x_{i} \leq 1$, there is a probability function $\prob$ such that $\prob(\phi) = 1 - \sum_{1 \leq i \leq n} x_{i}$, and such that $\prob(\gamma_{i}) = 1 - x_{i}$ for each $1 \leq i \leq n$.
\end{theorem}

\begin{proof}
 See the cited work.
\end{proof}

\begin{fact} \label{fact:setfmla-captured}
  If $\alpha \neq \{1\}$, if $\G$ is $\alpha$-satisfiable, $\delta$ is not a classical tautology, and there is no $\gamma \in \G$ such that $\sq{\gamma}{\delta}$ is classically valid, then $\sq{\G}{\delta}$ is not $\alpha$-preservation valid.
\end{fact}

\begin{proof}
  Consider all $\G' \subseteq \G$ such that $\sq{\G'}{\delta}$ is classically valid.
  Either some such $\G'$ is classically consistent, or all such $\G'$ are classically inconsistent.

  \begin{itemize}
    \item If there is some such $\G'$ that is classically consistent, take some $\G'' \subseteq \G'$ such that $\sq{\G''}{\delta}$ is classically valid and there is no $\G''' \subsetneq \G''$ with $\sq{\G'''}{\delta}$ classically valid.
          Since $\G'$ is classically consistent, $\G''$ must be as well.

          This meets the assumptions of \Cref{fact:adams-thm14}, so we apply that result.
          Where $|\G''| = n$, all that remains is to choose an appropriate sequence $x_{1}, \ldots, x_{m}$ with sum at most 1 to arrive at our desired $\alpha$-preservation counterexample.
          First, choose some $y \not\in \alpha$ such that $\frac{y + 1}{2} \in \alpha$.\footnote{This is always possible.
          If $\alpha = \hopen{0}{1}$, then let $y$ be 0; otherwise, let $x$ be the threshold of $\alpha$, and choose some positive $\varepsilon$ such that $\varepsilon < \min(x, 1 - x)$ and let $y = x - \varepsilon$.
          Then since $\varepsilon$ is positive we have $y < x$, and so $y \not \in \alpha$.
          To see that $\frac{y + 1}{2} \in \alpha$, note that $\frac{y + 1}{2} = \frac{x - \varepsilon + 1}{2} = x + \frac{- x - \varepsilon + 1}{2}$.
          This must be $> x$ since $\varepsilon < 1 - x$, so it's in $\frac{y + 1}{2} \in \alpha$.
          }

          By assumption, there is no $\G'$ with $|\G'| < 2$ such that $\sq{\G'}{\delta}$ is classically valid, so $|\G''| \geq 2$.
          Our desired sequence then has $x_{1} = x_{2} = 1 - \frac{y + 1}{2}$, and $x_{i} = 0$ for $i > 2$, if there are any.

          The cited theorem then assures us that there is a probability distribution $\prob$ such that $\prob(\gamma) \in \alpha$ for each $\gamma \in \G$, since for each such $\gamma$, either $\prob(\gamma) = \frac{y + 1}{2}$ or $\prob(\gamma) = 1$, and such that $\prob(\delta) = 1 - (2 - (y + 1)) = y$.
          By our choice of $y$, then, any probabilistic model with $\prob$ as its final coordinate is an $\alpha$-preservation counterexample to $\sq{\G}{\delta}$, as desired.

    \item If all such $\G'$ are classically inconsistent, we note that $\G$ is $\alpha$-satisfiable by assumption, so there is a $\prob$ such that $\prob(\gamma) \in \alpha$ for every $\gamma \in \G$.
          We then apply \Cref{fact:inconsistent-ms-to-zero} to conclude that there is some $\prob'$ that is an $\alpha$-preservation counterexample to $\sq{\G}{\delta}$, as desired.
  \end{itemize}
\end{proof}

\Cref{fact:setfmla-captured} drastically narrows down the ways a \setfmla\ argument $\sq{\G}{\phi}$ can come to be $\alpha$-preservation valid for $\alpha$ other than $\{1\}$---only if $\G$ itself is not $\alpha$-satisfiable, or $\delta$ is a classical tautology, or there is some single $\gamma \in \G$ such that $\sq{\gamma}{\delta}$ is classically valid.
There is no other way.
For example, it follows immediately from \Cref{fact:setfmla-captured} that $\sq{p, q}{p \land q}$ is $\alpha$-preservation invalid for all $\alpha \neq \{1\}$.\footnote{Recall from \Cref{sec:super-subv} that when $\alpha = \{1\}$ then $\alpha$-preservation consequence is exactly supervaluational, and this in turn is classical for \setfmla\ arguments, so this argument \emph{is} $\alpha$-preservation valid in that special case.}
After all, the set $\{p, q\}$ is $\alpha$-satisfiable for any $\alpha$, but $p \land q$ is not a classical tautology and neither $\sq{p}{p \land q}$ nor $\sq{q}{p \land q}$ is classically valid.

Essentially, what \Cref{fact:setfmla-captured} tells us, for \setfmla\ arguments in these preservation consequence relations, is this.
Premises can interact with each other (to reach $\alpha$-unsatisfiability), and any single premise can interact with the conclusion (to reach classical validity), and the conclusion alone might suffice (when it is a classical tautology), but: we can never have a case where validity is secured by premises interacting with each other \emph{and} the conclusion.
There is a slipperiness to each $\alpha$-preservation consequence, at least in its \setfmla\ fragment.

This can be extended directly to the full \setset\ framework, however:
\begin{corollary} \label{fact:setfmla-to-setset}
  If $\alpha \neq \{1\}$, if $\G$ is $\alpha$-satisfiable, $\bigvee \D$ is not a classical tautology, and if there is no $\gamma \in \G$ such that $\sq{\gamma}{\bigvee \D}$ is classically valid, then $\sq{\G}{\D}$ is not $\alpha$-preservation valid.
\end{corollary}

\begin{proof}
  From these assumptions, \Cref{fact:setfmla-captured} gives us that $\sq{\G}{\bigvee \D}$ is not $\alpha$-preservation valid.
  But any counterexample to $\sq{\G}{\bigvee \D}$ must at the same time be a counterexample to $\sq{\G}{\D}$.
\end{proof}

\Cref{fact:setfmla-captured} and \Cref{fact:setfmla-to-setset} are asymmetric in their assumptions, focusing on multiplicity among the \emph{premises} of an argument, and either assuming a single conclusion or else lumping all conclusions together with $\bigvee$.
As might be expected, these results can be dualized:

\begin{corollary} \label{fact:fmlaset-captured}
  If $\alpha \neq \hopen{0}{1}$, if $\D$ is not $\alpha$-tautologous, $\bigwedge \G$ is classically satisfiable, and if there is no $\delta \in \D$ such that $\sq{\bigwedge \G}{\delta}$ is classically valid, then $\sq{\G}{\D}$ is not $\alpha$-preservation valid.
\end{corollary}

\begin{proof}
  Suppose the antecedent.
  Then by \Cref{fact:tautology-satisfiability-dual} we have that $\neg \D$ is $\dalpha$-satisfiable.
  Moreover, since classical logic is self-dual (in the sense that $\sq{\Sigma}{\Theta}$ is classically valid iff $\sq{\neg \Theta}{\neg \Sigma}$ is), we know that there is no $\G' \subseteq \neg \D$ such that $|\G'| < 2$ and $\sq{\G'}{\neg \bigwedge \G}$ is classically valid.
  By a De Morgan equivalence, that ensures that there is no such $\G'$ with $\sq{\G'}{\bigvee \neg \G}$ classically valid.
  Thus, by \Cref{fact:setfmla-captured}, we can conclude that $\sq{\neg \D}{\neg \G}$ is not $\dalpha$-preservation valid, and so by \Cref{fact:tautology-satisfiability-dual} we have that $\sq{\G}{\D}$ is not $\alpha$-preservation valid, as desired.
\end{proof}

We pause for a moment to give a couple of examples and take stock.
\Cref{fact:preservation-is-subclassical}, \Cref{fact:setfmla-to-setset} and \Cref{fact:fmlaset-captured} together establish that a wide range of arguments are $\alpha$-preservation invalid.

For example, consider the argument $\sq{p, q \lor r}{p \land q, r}$.
This is classically valid, so \Cref{fact:preservation-is-subclassical} does not apply.
However, neither $\sq{p}{(p \land q) \lor r}$ nor $\sq{q \lor r}{(p \land q) \lor r}$ is classically valid, and $(p \land q) \lor r$ is not a classical tautology.
Thus, by \Cref{fact:setfmla-to-setset}, if $\alpha \neq \{1\}$ and $\{p, q \lor r\}$ is $\alpha$-satisfiable, the argument is $\alpha$-preservation invalid.
But $\{p, q \lor r\}$ is $\alpha$-satisfiable for any upset $\alpha$.
So this argument is $\alpha$-preservation invalid for all $\alpha \neq \{1\}$.
(Moreover, since the argument is also supervaluationistically invalid, by \Cref{fact:1super} it is also $\{1\}$-preservation invalid.)

On the other hand, consider the argument $\sq{p, q}{p \land q, p \land \neg q}$.
This is classically valid, so \Cref{fact:preservation-is-subclassical} does not apply.
Moreover, $\sq{p}{(p \land q) \vee (p \land \neg q)}$ is classically valid, so \Cref{fact:setfmla-to-setset} does not apply.
And $\sq{p \land q}{p \land q}$ is classically valid, so \Cref{fact:fmlaset-captured} does not apply.
The argument is both supervaluationistically and subvaluationistically valid, so $\alpha$-preservation valid for the extreme cases of $\alpha = \{1\}$ and $\alpha = \hopen{0}{1}$.
But what about non-extreme choices of $\alpha$?
Our foregoing results do not apply to this case.

Nonetheless, the argument is still $\alpha$ invalid for all such.
To see this, choose some $x, y \not \in \alpha$ such that $x + y \in \alpha$, with the constraint that $x + 2y \leq 1$.
(Note that this is not possible for the extreme thresholds $\hopen{0}{1}$ and $\{1\}$, but is possible for all other thresholds.)
Then consider the model $\tuple{\{a, b, c, d\}, \wp(\{a, b, c, d\}), \dn{\;}, \prob}$ where:
\begin{itemize}
  \item $\dn{p} = \{a, b\}$ and $\dn{q} = \{a, c\}$; and
  \item $\prob(\{a\}) = x$ and $\prob(\{b\}) = \prob(\{c\}) = y$ and $\prob(\{d\}) = 1 - (x + 2y)$
\end{itemize}
Now note that $\prob(p) = \prob(q) = x + y$, but that $\prob(p \land q) = x$ and $\prob(p \land \neg q) = y$; so we have an $\alpha$-preservation counterexample to this argument.

Our sufficient conditions for $\alpha$-preservation invalidity, then, are genuinely only sufficient: there remain $\alpha$-preservation invalid arguments that fall outside the purview of \Cref{fact:preservation-is-subclassical}, \Cref{fact:setfmla-to-setset}, \Cref{fact:fmlaset-captured}.
We conjecture a stronger claim:

\begin{conjecture}\label{conj:setset-captured}
  If $\hopen{0}{1} \neq \alpha \neq \{1\}$, if $\G$ is $\alpha$-satisfiable and $\D$ is not $\alpha$-tautologous, and if there are no $\gamma \in \G, \delta \in \D$ such that $\sq{\gamma}{\delta}$ is classically valid, then $\sq{\G}{\D}$ is not $\alpha$-preservation valid.
\end{conjecture}

If \Cref{conj:setset-captured} is true, it is also comprehensive: every $\alpha$-preservation invalid argument would satisfy its assumptions (since failing any one of the assumptions is immediately sufficient for $\alpha$-preservation validity).
We know of no $\alpha$-preservation valid argument that provides a counterexample to \Cref{conj:setset-captured}, but we also have not been able to prove the conjecture.
For now, then, we leave things where they stand: with some broad sufficient conditions for invalidity in place.

\section{Properties of Preservation Consequence}
%Some other results about preservation validity}
\label{sec:pres-properties}

We can now state some general properties of $\alpha$-consequence relations: we start with an overview of relations between $\alpha$-preservation relations for different $\alpha$, then establish how many distinct consequence relations there are, and conclude with a note on operational and structural features of those relations.

\subsection{Containment relations}

With these results established, we are in a place to note some other interesting features of $\alpha$-preservation validity. The first is a slight strengthening of a result of \cite{paris2004deriving}.
The strengthening is just that we consider both closed and open upsets, while that paper considers (in effect) only closed.
Our proof of this result, however, is quite different, since we can appeal to the foregoing.

\begin{fact} \label{fact:full-narrowing-preserves-setfmla}
  If $\beta \subseteq \alpha$, then if $\sq{\G}{\phi}$ is $\alpha$-preservation valid, it is $\beta$-preservation valid.
\end{fact}

\begin{proof}
  Suppose $\beta \subseteq \alpha$ and $\sq{\G}{\phi}$ is $\alpha$-preservation valid.
  Then by \Cref{fact:setfmla-captured}, either $\phi$ is a classical tautology, or there is some $\gamma \in \G$ such that $\sq{\gamma}{\phi}$ is classically valid, or $\G$ is $\alpha$-inconsistent.
  If any of these holds, however, then $\sq{\G}{\phi}$ is $\beta$-preservation valid as well.
  (For the third disjunct: if $\G$ is $\alpha$-inconsistent, then it must be $\beta$-inconsistent, and so $\sq{\G}{\phi}$ is $\beta$-preservation valid.)
\end{proof}

That is, as we narrow the upsets we consider, we strengthen the \setfmla\ fragment of the preservation consequence that results.
As you might expect, the situation is the reverse for \fmlaset\ arguments: the \fmlaset\ fragment of these preservation consequence relations gets stronger as the upset \emph{widens}.

\begin{fact} \label{fact:full-widening-preserves-fmlaset}
  If $\alpha \subseteq \beta$, then if $\sq{\phi}{\D}$ is $\alpha$-preservation valid, it is $\beta$-preservation valid.
\end{fact}

\begin{proof}
  Suppose $\alpha \subseteq \beta$ and $\sq{\phi}{\D}$ is $\alpha$-preservation valid.
  Then by \Cref{fact:fmlaset-captured}, either $\phi$ is a classical contradiction, or there is some $\delta \in \D$ such that $\sq{\phi}{\delta}$ is classically valid, or $\D$ is $\alpha$-tautologous.
  If any of these holds, however, then $\sq{\phi}{\D}$ is $\beta$-preservation valid as well.
\end{proof}

\Cref{fact:full-narrowing-preserves-setfmla,fact:full-widening-preserves-fmlaset} together give us an interesting overview of these preservation consequences: in their $\setfmla$ fragment, they get stronger as the upset narrows, but in their $\fmlaset$ fragment, they get stronger as the upset widens instead.
As we saw in \Cref{sec:super-subv}, the limits are the familiar relations of supervaluational (at $\{1\}$) and subvaluational (at $\hopen{0}{1}$) consequence.
This also means that every preservation consequence relation agrees with every other on \fmlafmla\ arguments, regardless of which upset is chosen, since \fmlafmla\ arguments are both \setfmla\ and \fmlaset (and so every preservation consequence relation is classical on \fmlafmla\ arguments).

\Cref{fact:full-narrowing-preserves-setfmla} also allows us to connect our results here to a fascinating result of \cite{knight2003probabilistic}, connecting the \setfmla\ fragment of subvaluationistic logic to a related framework.
Note that by \Cref{fact:full-narrowing-preserves-setfmla} and \Cref{fact:01sub}, we know that the \emph{intersection} of all the $\alpha$-preservation consequence relations has as its \setfmla\ fragment the \setfmla\ fragment of subvaluationistic logic.
\cite{knight2003probabilistic} shows, in effect, that the intersection of a much wider class of \setfmla\ probabilistic consequence relations also matches subvaluationistic logic.

As for the more general \setset\ situation, these neat ordering facts do not obtain.
For example, consider again the extremes of $\{1\}$-preservation and $\hopen{0}{1}$-preservation.
As mentioned before, $\sq{p, q}{p \land q}$ is valid in supervaluationist logic but not subvaluationist logic, while $\sq{p \lor q}{p, q}$ is valid in subvaluationist logic but not supervaluationist logic.

\subsection{Counting preservation relations}
%Identity and distinctness}
\label{sec:how-many-consequence}

The purpose of this section is to show that there are continuum many distinct $\alpha$-preservation consequence relations, and to consider the relationship between the closed and open upsets at any particular threshold.

First, we show that we can get our hands on enough sentences with the right logical properties for our purposes:

\begin{fact} \label{fact:pairwise-inconsistent}
  For every natural number $m$, there is a set of $m$ sentences such that each of them is classically consistent and any pair of them is classically inconsistent.
\end{fact}

\begin{proof}
  Take $n$ atomic sentences, where $n$ is such that $2^{n} \geq m$; let these be $\{p_{1}, \ldots, p_{n}\}$.
  There are $2^{n}$ sentences of the form $\pm p_{1} \land \ldots \ldots \pm p_{n}$, where $+p_{i}$ is $p_{i}$ and $- p_{1}$ is $\neg p_{i}$, and these are pairwise inconsistent, since any two differ on at least one $p_{i}$.
  They are also all individually consistent.
  So take any $m$ of them.
\end{proof}

With that in hand, we move to a lemma that allows us to draw a divide between the closed and open upsets at any rational number:

\begin{lemma} \label{fact:every-rational-maxsat}
  For every rational number $\frac{n}{m} \in (0, 1)$, there is a set $\G$ that is $\clos{\frac{n}{m}}{1}$-satisfiable and not $\hopen{\frac{n}{m}}{1}$-satisfiable.\footnote{Many thanks to Peter Fritz for helping to find this proof.}
\end{lemma}

\begin{proof}
   By \Cref{fact:pairwise-inconsistent}, there is a set of $m$ sentences that are individually consistent and pairwise inconsistent; let these be $\phi_{1}, \ldots, \phi_{m}$.
   Our target $\G$ is $\abst{\phi_{i_{1}} \lor \ldots \lor \phi_{i_{n}}}{1 \leq i_{1} < \ldots < i_{n} \leq m}$.

   First, to see that $\G$ is $\clos{\frac{n}{m}}{1}$-satisfiable.
   Consider the model $\tuple{W, \alg, \dn{\;}, \prob}$, where $W = \{w_{1}, \ldots, w_{m}\}$; where $\alg = \wp(W)$; where $\dn{\;}$ is such that $\dn{\phi_{i}} = \{w_{i}\}$, which is possible since each $\phi_{i}$ is consistent and none classically entails any other; and where $\prob(\{w_{i}\}) = \frac{1}{m}$.
   In this model, $\prob(\gamma) = \frac{n}{m}$ for each $\gamma \in \G$.

   Now, to see that $\G$ is not $\hopen{\frac{n}{m}}{1}$-satisfiable.
   Suppose otherwise; then we have a probability distribution $\prob$ such that $\prob(\gamma) > \frac{n}{m}$ for each $\gamma \in \G$.
   Note that since each $\gamma \in \G$ is an $n$-ary disjunction of pairwise incompatible $\phi_{i}$s, and since each choice of $n$ $\phi_{i}$s is disjoined in some such $\gamma$, the probabilities of the sentences in $\G$ are exactly the $n$-ary sums whose addends are drawn (without replacement) from the $\prob(\phi_{i})$s.

   Since by supposition all such $n$-ary sums are $> \frac{n}{m}$, there can be at most $n - 1$ of the $\phi_{i}$s such that $\prob(\phi_{i}) \leq \frac{1}{m}$.
   So there are at least $m - n + 1$ of the $\phi_{i}$ that are such that $\prob(\phi_{i}) > \frac{1}{m}$; choose $m - n$ of these.
   For concreteness, but without loss of generality, we suppose that these are $\phi_{1}, \ldots, \phi_{m - n}$.

   Now, $\prob(\phi_{1}) + \ldots + \prob(\phi_{m - n}) > \frac{m - n}{m}$, and so $\prob(\phi_{m - n + 1}) + \ldots + \prob(\phi_{m}) < \frac{n}{m}$, since the sum of these two sums cannot be greater than $1$.
   But the latter sum is an $n$-ary sum whose addends are drawn without replacement from the $\prob(\phi_{i})$.
   By assumption all such sums are $> \frac{n}{m}$, so we have a contradiction.
\end{proof}

We also note a result of \cite{knight2002measuring} about $\alpha$-satisfiability; we use both this and \Cref{fact:every-rational-maxsat} both here and later in \Cref{sec:symm-cons}:

\begin{theorem}[\protect{\citealp[Thm.\ 4.14, p.\ 86]{knight2002measuring}}]\label{knight-thm414}
  For any finite set $\G$ of sentences, there is some maximum $x$ such that $\G$ is $\clos{x}{1}$-satisfiable, and such maximum $x$ is rational.
\end{theorem}

\begin{proof}
 See the cited work.
\end{proof}

\noindent This then gives a quick corollary, which is the form we'll use it in:

\begin{corollary} \label{fact:sat-rational}
  For any $x \in (0, 1)$, if there is a finite set $\G$ such that $\G$ is $\clos{x}{1}$-satisfiable and $\hopen{x}{1}$-unsatisfiable, then $x$ is rational.
\end{corollary}

\begin{proof}
 Immediate from \Cref{knight-thm414}.
\end{proof}

This is now enough for us to summarize the situation surrounding any $\clos{x}{1}$- and $\hopen{x}{1}$-preservation consequence.
For purposes of this discussion, we treat these consequence relations simply as sets of valid arguments; we call them \demph{distinct} when they are distinct as sets and \demph{incomparable} when neither set is a subset of the other.

\begin{theorem} \label{fact:distinct-iff-rational}
  For any $x \in (0, 1)$, if the consequence relations of $\clos{x}{1}$-preservation and $\hopen{x}{1}$-preservation are distinct, then $x$ is rational; and if $x$ is rational, then the consequence relations of $\clos{x}{1}$-preservation and $\hopen{x}{1}$-preservation are incomparable.
\end{theorem}

\begin{proof}
  First, suppose the consequence relations are distinct, to show that $x$ is rational.
  Since $\clos{x}{1}$-preservation consequence and $\hopen{x}{1}$-preservation consequence differ, there must be some argument $\sq{\G}{\D}$ that is either $\hopen{x}{1}$-preservation valid and $\clos{x}{1}$-preservation invalid, or $\clos{x}{1}$-preservation valid and $\hopen{x}{1}$-preservation invalid,

  Suppose the first disjunct.
  Then by \Cref{fact:setfmla-to-setset}, it must be that $\G$ is $\clos{x}{1}$-satisfiable but not $\hopen{x}{1}$-satisfiable.
  So, by \Cref{fact:sat-rational}, $x$ is rational.

  On the other hand, suppose the second disjunct.
  Then by \Cref{fact:fmlaset-captured}, it must be that $\D$ is $\clos{x}{1}$-tautologous but not $\hopen{x}{1}$-tautologous.
  By \Cref{fact:tautology-satisfiability-dual}, then, $\neg \D$ is $\clos{1 - x}{1}$-satisfiable but not $\hopen{1 - x}{1}$-satisfiable.
  By \Cref{fact:sat-rational}, then, $1 - x$ is rational; and then so is $x$ itself.

  Next, we show that when $x$ is rational the relations are incomparable.
  \Cref{fact:every-rational-maxsat} assures us that when $x$ is rational there is a set $\G$ that is $\clos{x}{1}$-satisfiable but not $\hopen{x}{1}$-satisfiable.
  But then $\sq{\G}{\emptyset}$ is an argument that is $\hopen{x}{1}$-preservation valid and not $\clos{x}{1}$-preservation valid.

  Also, since $x$ is rational so is $1 - x$, and there is thus a set $\D$ that is $\clos{1 - x}{1}$-satisfiable but not $\hopen{1 - x}{1}$ satisfiable, again by \Cref{fact:every-rational-maxsat}.
  But then by \Cref{fact:tautology-satisfiability-dual} we have that $\neg \D$ is $\clos{x}{1}$-tautologous and not $\hopen{x}{1}$-tautologous, which is to say that $\sq{\emptyset}{\neg \D}$ is $\clos{x}{1}$-preservation valid and $\hopen{x}{1}$-preservation invalid.
\end{proof}

We can also now count the distinct $\alpha$-preservation consequence relations:

\begin{theorem} \label{fact:uncountable-closed}
  For any distinct $x, y \in \hopen{0}{1}$, if $x$ is the threshold of $\alpha$ and $y$ is the threshold of $\beta$, then $\alpha$-preservation consequence and $\beta$-preservation consequence are incomparable.
\end{theorem}

\begin{proof}
  Without loss of generality, let $x < y$; we first show there is a set that is $\clos{x}{1}$-satisfiable but not $\clos{y}{1}$-satisfiable.
  To see this, take some rational $z$ such that $x < z < y$, and use \Cref{fact:every-rational-maxsat} to arrive at some $\G$ that is $\clos{z}{1}$-satisfiable but not $\hopen{z}{1}$-satisfiable.
  Since $\G$ is $\clos{z}{1}$-satisfiable and $x < z$, it is also $\hopen{x}{1}$-satisfiable, and so $\alpha$-satisfiable.
  And since $\G$ is not $\hopen{z}{1}$-satisfiable and $z < y$, it is also not $\clos{y}{1}$-satisfiable, and so not $\beta$-satisfiable.
  Now, consider the argument $\sq{\G}{\emptyset}$; this is $\beta$-preservation valid and $\alpha$-preservation invalid.

  For the other direction of incomparability, note that since $x < y$ we also have $1 - y < 1 - x$; so we can take some rational $w$ such that $1 - y < w < 1 - x$, and use \Cref{fact:every-rational-maxsat} to arrive at some $\D$ that is $\clos{w}{1}$-satisfiable but $\hopen{w}{1}$-unsatisfiable.
  This ensures that $\D$ is $\hopen{1 - y}{1}$-satisfiable and $\clos{1 - x}{1}$-unsatisfiable,  which by \Cref{fact:tautology-satisfiability-dual} gives us that $\neg \D$ is $\hopen{x}{1}$-tautologous but not $\clos{y}{1}$-tautologous.
  This ensures that $\neg \D$ is $\alpha$-tautologous and not $\beta$-tautologous, and so $\sq{\emptyset}{\neg \D}$ is $\alpha$-preservation valid and not $\beta$-preservation valid.
\end{proof}

\begin{corollary} \label{fact:uncountably-many-preservation}
  There are uncountably many distinct preservation consequence relations, and any distinct preservation consequence relations are incomparable.
\end{corollary}

\begin{proof}
  The uncountability is from \Cref{fact:uncountable-closed} and the fact that $\hopen{0}{1}$ is uncountable.

  For the second half, take any upsets $\alpha, \beta$.
  If $\alpha$ and $\beta$ have the same threshold, then \Cref{fact:distinct-iff-rational} ensures that if $\alpha$-preservation consequence and $\beta$-preservation consequence are distinct they are incomparable.
  On the other hand, if $\alpha$ and $\beta$ have distinct thresholds, then \Cref{fact:uncountable-closed} ensures that $\alpha$-preservation consequence and $\beta$-preservation consequence are certainly incomparable.
\end{proof}

The results of this section give us an interesting picture.
In some sense the fact that \cite{knight2002measuring}, \cite{paris2004deriving}, and others consider only closed upsets makes a very small difference---one that matters at every rational threshold, but nowhere else.
On the other hand, the fact that \cite{paris2004deriving} considers only the \setfmla\ framework makes a very big difference: \Cref{fact:full-narrowing-preserves-setfmla} gives a linear order to the \setfmla\ fragments of these consequence relations, while \Cref{fact:uncountably-many-preservation} shows that in the full \setset\ framework this very much does not obtain.

\subsection{Structural and operational features}
\label{sec:disc-pres}

How different and how similar are preservation consequence relations from classical consequence, when considering structural and operational properties?

First of all, and for comparison with what comes next, we note that, simply because our $\alpha$-preservation consequence relations are all defined by preservation of some status, all of them have the Tarskian properties of reflexivity, monotonicity, and transitivity, in the following senses, by \citet[Thm.\ 2.1]{ss:mcl}:
\begin{itemize}
  \item $\sq{\phi}{\phi}$ is valid;
  \item if $\sq{\G}{\D}$ is valid, then $\sq{\G', \G}{\D, \D'}$ is valid; and
  \item if $\sq{\G}{\D, \phi}$ and $\sq{\phi, \G}{\D}$ are valid, then $\sq{\G}{\D}$ is valid.\footnote{\cite{ss:mcl} claims a different transitivity property here, one that is properly stronger than ours in general, but is equivalent in our present setting where only finite arguments are considered.
        See \citet[\S 2.1]{ss:mcl} and \citet[\S 1.16]{humberstone:c} for helpful discussion of this difference; or for further discussion of various properties that travel under the name `transitivity', and proofs of relations between these properties, see \cite{ripley:otocr,chen:TransitivityLogicalConsequence2024}.}
\end{itemize}

Regarding operational properties, we note (also for comparison with other consequence relations to be discussed presently) the situation around $n$-ary conjunction introduction.
Let $p_{1}, \ldots, p_{n}$ be the first $n$ atomic sentences, and let $CI_{n}$ be the argument $\sq{p_{1}, \ldots, p_{n}}{\bigwedge p_{i}}$.
Since this is a \setfmla\ classically valid argument, it is supervaluationistically valid, and so $\{1\}$-preservation valid.
Moreover, $CI_{0}$ and $CI_{1}$ are $\alpha$-preservation valid for any $\alpha$.
However, for any $\alpha \neq \{1\}$ and any $n \geq 2$, the argument $CI_{n}$ is $\alpha$-preservation invalid.
This follows immediately from \Cref{fact:setfmla-captured}, but we note it here for later comparison.

Next, we recall that subvaluationistic logic is what is sometimes called \demph{weakly paraconsistent} (see \citealt{hyde1997heaps}): while $\sq{\phi \land \neg \phi}{\psi}$ is subvaluationistically valid for any sentences $\phi, \psi$, there are nonetheless choices of $\phi, \psi$ where $\sq{\phi, \neg \phi}{\psi}$ is invalid.
For example, $\sq{p, \neg p}{q}$ is subvaluationistically invalid.
Upsets $\alpha$ that do not include $.5$ do not determine preservation consequence relations that are paraconsistent in any sense.
Conversely, supervaluational logic has been called \textit{weakly paracomplete} (see \citealt[p. 76]{hyde2008vagueness}, crediting Arruda): while $\sq{\psi}{\phi \vee \neg \phi}$ is supervaluationistically valid, $\sq{\psi}{\phi, \neg \phi}$ is not.

For later reference, we note that it is not just the extremes; indeed every preservation consequence relation has exactly one of these properties:
\begin{fact} \label{fact:weak-para-exactly-one}
  For any $\alpha$, if $.5 \not \in \alpha$, then $\alpha$-preservation consequence is weakly paracomplete but not weakly paraconsistent; and if $.5 \in \alpha$, then $\alpha$-preservation consequence is weakly paraconsistent but not weakly paracomplete.
\end{fact}

\begin{proof}
  Immediate, once it's noted that there is a model that assigns probability $.5$ to both $p$ and $\neg p$, and that there is no model that assigns a probability strictly greater than $.5$ to both $p$ and $\neg p$.
\end{proof}

This also allows us to see that no $\alpha$-preservation consequence relation is self-dual.
In \Cref{sec:alpha-satisf-alpha}, we noted that no upset can be self-dual, since any upset $\alpha$ includes $.5$ iff its dual $\dalpha$ does not.
However, that left open whether there could be two dual upsets which, while necessarily distinct as sets, still manage to determine the same preservation consequence relation.
We can now see that this is never the case, since for any upset $\alpha$, exactly one of $\alpha$-preservation consequence and $\dalpha$-preservation consequence must be weakly paraconsistent, and so they cannot be the same.

\section{Symmetric consequence}
\label{sec:symm-cons}

We find it natural to think that classical logic should in some sense be a limit case of probabilistic reasoning, applying perfectly in cases of perfect certainty, and gradually approached as levels of certainty increase.
However, as we've seen, neither material consequence nor preservation consequence seems to be able to support this natural thought.
Material consequence is fully classical regardless of what upset we choose, and so regardless of the level of certainty in play; classical consequence gets to apply perfectly in cases of perfect certainty, but it applies perfectly in all other cases as well!
It is not in any illuminating sense \emph{approached} as certainty increases.

On the other hand, preservation consequence does at least provide us with some distinct consequence relations for different upsets $\alpha$, as described in \Cref{fact:distinct-iff-rational} and \Cref{fact:uncountably-many-preservation}.
However, it also fails to support our natural thought, in two ways.
First, at the limit of perfect certainty---the upset $\{1\}$---preservation consequence is not classical but supervaluational.

But even if we were to artificially restrict our attention to the \setfmla\ fragment, where this difference is not visible, \Cref{fact:setfmla-captured} tells us that again, we do not \emph{approach} this limit as our upset narrows in any useful or informative way.
To see this, take any classically-valid \setfmla\ argument whose premises are classically consistent, whose conclusion is not classically tautologous, and where no single premise classically entails the conclusion.
For example, take modus ponens $\sq{p \hook q, p}{q}$, or conjunction introduction $\sq{p, q}{p \land q}$, or really almost any example that actually gets used in argumentation.
All such arguments, because \setfmla\ and classically valid, are $\{1\}$-preservation valid.
But by \Cref{fact:setfmla-captured}, none of them is $\alpha$-preservation valid for any \emph{other} choice of $\alpha$.
These arguments all languish in invalidity as our choice of $\alpha$ narrows until the very last instant, at $\{1\}$, where they all simultaneously leap to validity.
This is not the gradual approach envisioned in our natural thought.

In this section, then, we introduce a third counterexample notion, determining what we call \demph{symmetric consequence}.
Unlike material and preservation consequence, symmetric consequence \emph{does} support our natural thought, as we will show.

\subsection{Definition}
\label{sec:sc-definition}

To define symmetric consequence, we use the idea of the \demph{mirror image} of an upset.
Recall from \Cref{sec:alpha-satisf-alpha} that for any upset $\alpha$, its mirror image $\malpha$ is $\abst{x}{1 - x \in \alpha}$.
This reflects $\alpha$ around the midpoint $.5$.
Since $\alpha$ is an upset, this means that $\malpha$ must contain 0, must not contain 1, and must be closed \emph{downwards}.

With this notion in hand, we are ready to define symmetric counterexamples:
\begin{defn} \label{defn:sym-con}
A probabilistic model $\modl_{\prob}$ is an $\alpha$-symmetric counterexample to an argument $\sq{\G}{\D}$ iff $\prob[\G] \subseteq \alpha$ and $\prob[\D] \subseteq \malpha$.
Thus, the argument $\sq{\G}{\D}$ is $\alpha$-symmetric valid iff every $\modl_{\prob}$ is such that if $\prob[\G] \subseteq \alpha$, then there is some $\delta \in \D$ with $\prob(\delta) \not\in \malpha$.
\end{defn}

In the same way that an upset $\alpha$ indicates probabilities that are high enough, we take its mirror image $\malpha$ to indicate probabilities that are too low.
By using the mirror image in this way, we assume that tight standards for what counts as high enough come linked with tight standards for what counts as too low, and similarly for loose standards.\footnote{It is, of course, possible to consider a freer notion that would allow these standards to be set independently of each other.
  Sticking to probability distributions, this freer notion is closely related to ideas in \cite{knight2003probabilistic,paris2009inconsistency}.
  To see a similar idea in the case of fuzzy logic see \cite{cobreros2024tolerance}.
  For more general reflections on such independence between premises and conclusions, see \cite{humberstone:hl,bmw:imv,ces:clcmvl,fr:valuations,chemla2019suszko}.
  Here, though, we keep things simple and do not move to such freer settings.}

For example, consider again (as in \Cref{sec:pres-cons}) the upset $\alpha = \hopen{.7}{1}$ and a single roll of a 6-sided die, where $p$ says that the die comes up strictly greater than $1$ and $q$ says that the die comes up  strictly less than $6$.
To give an $\alpha$-symmetric counterexample to the argument $\sq{p, q}{p \land q}$, we would need some probability distribution $\prob$ such that $\prob(p) > .7$ and $\prob(q) > .7$ and $\prob(p \land q) < .3$.
But this can never happen, whether or not the die is fair; by the \frech\ bounds (which are a special case of \citealt[Thm.\ 13, p.\ 38]{adams1998primer}), for any probability distribution $\prob$ where $\prob(p) > .7$ and $\prob(q) > .7$, we have $\prob(p \land q) > .4$.
So the argument $\sq{p, q}{p \land q}$ is $\hopen{.7}{1}$-symmetric valid.

As we saw above, however, this same argument is not $\hopen{.7}{1}$-preservation valid.
Intuitively, in this argument probabilities can slip from high enough to not high enough as we go from premises to conclusion, but they cannot slip from high enough to too low.

\subsubsection{Negation and monotonicity}
\label{sec:negat-symm-cons}

We begin our consideration of $\alpha$-symmetric consequence by quickly noting some features it has that will smooth the reasoning to follow.

First, in \Cref{sec:disc-pres}, we noted that no $\alpha$-preservation consequence relation is self-dual, in the sense that it validates any $\sq{\G}{\D}$ iff it validates $\sq{\neg \D}{\neg \G}$.
The situation is very different for our symmetric consequence relations:

\begin{fact} \label{fact:sym-negation}
  For any upset $\alpha$: the argument $\sq{\phi, \G}{\D}$ is $\alpha$-symmetrically valid iff $\sq{\G}{\D, \neg \phi}$ is; and the argument $\sq{\G}{\D, \phi}$ is $\alpha$-symmetrically valid iff $\sq{\neg \phi, \G}{\D}$ is.
\end{fact}

\begin{proof}
  Immediate from \Cref{defn:mirror-dual,defn:sym-con}, recalling that $\prob(\neg \phi) = 1 - \prob(\phi)$ for any probability distribution $\prob$.
\end{proof}

\Cref{fact:sym-negation} more than suffices to show that every $\alpha$-symmetric consequence relation is self-dual, regardless of the choice of $\alpha$.

Second (and relatedly), negation gives us a direct bridge between $\alpha$-symmetric validity and $\alpha$-unsatisfiability:

\begin{fact} \label{fact:sym-unsat}
  An argument $\sq{\G}{\D}$ is $\alpha$-symmetric valid iff $\G \cup \neg \D$ is $\alpha$-unsatisfiable.
\end{fact}

\begin{proof}
 Immediate from \Cref{defn:sat-taut,defn:mirror-dual,defn:sym-con}, as in the proof of \Cref{fact:sym-negation}.
\end{proof}

\noindent In what follows, we will appeal to \Cref{fact:sym-unsat} without further comment, treating questions of $\alpha$-symmetric validity and of $\alpha$-unsatisfiability interchangeably.

Third, we turn to the question of the monotonicity of $\alpha$-symmetric consequence.
(We leave consideration of the other Tarskian properties for \Cref{sec:sym-tarskian}, since we will not need them before then.)

\begin{fact} \label{fact:sym-monotonic}
  For any upset $\alpha$, the $\alpha$-symmetric consequence relation is monotonic.
\end{fact}

\begin{proof}
 Suppose that $\sq{\G', \G}{\D, \D'}$ is $\alpha$-symmetric invalid; then it has a counterexample.
 But any counterexample to this argument is also a counterexample to $\sq{\G}{\D}$.
\end{proof}

\subsection{Classicality in the limit}
\label{sec:classicality-limit}

Here, we show that $\alpha$-symmetric consequence fits well with our natural thought above.
That is, as we move from wider upsets to narrower, the resulting symmetric consequence relations steadily approach classical consequence until, at the narrowest upset $\{1\}$, they reach it exactly.

\begin{fact} \label{fact:sym-order}
  If $\alpha \subseteq \beta$ and $\sq{\G}{\D}$ is $\beta$-symmetric valid, then it is $\alpha$-symmetric valid.
\end{fact}

\begin{proof}
  We show the contrapositive; suppose that we have an $\alpha$-symmetric counterexample to $\sq{\G}{\D}$.
  Since $\alpha \subseteq \beta$, we know that $\malpha \subseteq \mirror{\beta}$, and so this very counterexample is itself a $\beta$-symmetric counterexample as well.
\end{proof}

Compare \Cref{fact:sym-order} to \Cref{fact:full-narrowing-preserves-setfmla,fact:full-widening-preserves-fmlaset}.
For preservation consequence, narrowing our upset can have complex results: it can move \setfmla\ arguments from invalid to valid, and can move \fmlaset\ arguments in the other direction, from valid to invalid; and can affect other arguments in either direction.\footnote{Consider, for example, the argument $\sq{p, q, \neg(p \land q)}{r, \neg r}$.
This is $\alpha$-preservation valid because of its conclusion set if $.5 \in \alpha$, and it is $\alpha$-preservation valid because of its premise set if $\frac{2}{3} \not\in \alpha$.
However, for any $\alpha$ where $.5 \not\in \alpha$ and $\frac{2}{3} \in \alpha$, this argument is $\alpha$-preservation invalid, as we can build a model that assigns probability $\frac{2}{3}$ to each premise and $.5$ to each conclusion.
So as we consider the full range of upsets, narrowing from $\hopen{0}{1}$ to $\{1\}$, this argument goes from valid to invalid and then back to valid.}
For symmetric consequence, on the other hand, narrowing our upset can only move an argument from invalid to valid, full stop.

This drastically simplifies the situation, by giving us a linear order on $\alpha$-symmetric consequence relations, with $\hopen{0}{1}$-symmetric consequence as the weakest and $\{1\}$-symmetric consequence as the strongest.
We turn now to describing these two consequence relations.

\begin{fact} \label{fact:sym-1-classical}
  An argument $\sq{\G}{\D}$ is $\{1\}$-symmetric valid iff it is classically valid.
\end{fact}

\begin{proof}
  First, suppose that we have a classical counterexample to $\sq{\G}{\D}$.
  Then consider the model $\tuple{\{w\}, \wp(\{w\}), \dn{\;}, \prob}$, where $\dn{\;}$ is set up to make the world $w$ that classical counterexample.
  This gives us $\prob(\gamma) = 1$ for each $\gamma \in \G$ and $\prob(\delta) = 0$ for each $\delta \in \D$, so we have our $\{1\}$-symmetric counterexample.

  For the other direction, suppose we have a $\{1\}$-symmetric counterexample to $\sq{\G}{\D}$.
  This is some $\tuple{W, \alg, \dn{\;}, \prob}$ where $\prob(\gamma) = 1$ for each $\gamma \in \G$ and $\prob(\delta) = 0$ for each $\delta \in \D$.
  Since $\G$ and $\D$ are finite, this means that $\prob(\bigwedge \G) = 1$ and $\prob(\bigvee \D) = 0$ as well.
  Since $\prob(\bigwedge \G) > \prob(\bigvee \D)$, we must have $\dn{\bigwedge \G} \not\subseteq \dn{\bigvee \D}$, so there must be some $w \in W$ where $w \in \dn{\bigwedge \G}$ and $w \not\in \dn{\bigvee \D}$; this $w$ is our classical counterexample.
\end{proof}

As for $\hopen{0}{1}$-symmetric validity, it turns out to be relatively simple, obtaining exactly when some premise is classically contradictory or some conclusion is classically tautologous:

\begin{fact}
  The argument $\sq{\G}{\D}$ is $\hopen{0}{1}$-symmetric valid iff either: there is some $\gamma \in \G$ such that $\gamma$ is classically contradictory; or there is some $\delta \in \D$ such that $\delta$ is a classical tautology.
\end{fact}

\begin{proof}
  Since $\prob(\gamma) = 0$ for all probability distributions $\prob$ and classically unsatisfiable $\gamma$, and similarly $\prob(\delta) = 1$ for all probability distributions $\prob$ and classically tautologous $\delta$, the right-to-left direction is secured.

  For the left-to-right direction, note that a $\hopen{0}{1}$-symmetric counterexample to $\sq{\G}{\D}$ is any model that assigns non-0 probability to everything in $\G$ and non-1 probability to everything in $\D$.
  Now, we proceed contrapositively.

  Take some $\sq{\G}{\D}$ with no classically unsatisfiable $\gamma \in \G$ and no classically tautologous $\delta \in \D$.
  If $\G \cup \D = \emptyset$, then any model is a $\hopen{0}{1}$-symmetric counterexample to $\sq{\G}{\D}$, so we have invalidity; assume, then, that $\G \cup \D$ is nonempty.
  Let $\G = \{\gamma_{1}, \ldots, \gamma_{m}\}$ and $\D = \{\delta_{1}, \ldots, \delta_{n}\}$; we know $m + n \geq 1$.

  Now, take a model $\tuple{W, \alg, \dn{\;}, \prob}$ such that:
  \begin{itemize}
    \item $W = \{w_{1}, \ldots, w_{m}, w_{m+1}, \ldots, w_{m+n}\}$;
    \item $\alg = \wp(W)$;
    \item $\dn{\;}$ is such that:
          \begin{itemize}
            \item for all $1 \leq i \leq m$, we have $w_{i} \in \dn{\gamma_{i}}$; and
            \item for all $1 \leq j \leq n$, we have $w_{m + j} \not\in \dn{\delta_{j}}$; and
          \end{itemize}
    \item for all $1 \leq k \leq m + n$, we have $\prob(\{w_{k}\}) = \frac{1}{m + n}$.
  \end{itemize}

  Such a model exists; the constraints on $\dn{\;}$ are jointly achievable because each constrains a different world, and we know that each $\gamma_{i}$ is classically satisfiable and each $\delta_{j}$ classically nontautologous.
  In this model, we have $\prob(\gamma_{i}) \geq \frac{1}{m + n} > 0$ for each $\gamma_{i} \in \G$, and $\prob(\delta_{j}) \leq 1 - \frac{1}{m + n} < 1$ for each $\delta_{j} \in \D$, so the model is a $\hopen{0}{1}$-symmetric counterexample to $\sq{\G}{\D}$.
\end{proof}

Our picture of $\alpha$-symmetric consequence is filling in bit by bit: at the narrowest upset $\{1\}$, we indeed reach exactly \setset\ classical logic, but at the widest upset $\hopen{0}{1}$, we have something different. Symmetric consequence is thus unlike both preservation consequence (which never gives classical logic) and material consequence (which never gives anything else).

We can draw on some of our earlier reasoning around $\alpha$-satisfiability (from \Cref{sec:how-many-consequence}) to learn about what happens along the way, and in particular which distinct upsets determine distinct symmetric consequence relations.

\begin{fact} \label{fact:sym-distinct-iff-rational}
  For any $x \in (0, 1)$, the $\clos{x}{1}$-symmetric and $\hopen{x}{1}$-symmetric consequence relations are distinct iff $x$ is rational.
\end{fact}

\begin{proof}
  Directly from \Cref{fact:sat-rational} and \Cref{fact:every-rational-maxsat}.
\end{proof}

\begin{fact} \label{fact:sym-closed-distinct}
  For any distinct $x, y \in \hopen{0}{1}$, the $\clos{x}{1}$-symmetric and $\clos{y}{1}$-symmetric consequence relations are distinct, and so there are uncountably many distinct symmetric consequence relations.
\end{fact}

\begin{proof}
  As in \Cref{fact:uncountable-closed}.
  Without loss of generality, let $x < y$.
  Take some rational $z$ such that $x \leq z < y$, and use \Cref{fact:every-rational-maxsat} to arrive at some $\G$ that is $\clos{z}{1}$-satisfiable but not $\hopen{z}{1}$-satisfiable.
\end{proof}

There is more to say, however, about $\alpha$-symmetric consequence for closed upsets $\alpha$:

\begin{fact} \label{fact:limit}
  If $\sq{\G}{\D}$ is $\clos{x}{1}$-symmetric valid, then there is some $\alpha \supsetneq \clos{x}{1}$ such that $\sq{\G}{\D}$ is $\alpha$-symmetric valid.
\end{fact}

\begin{proof}
  Take any $\clos{x}{1}$-symmetric valid argument $\sq{\G}{\D}$.
  By \cref{knight-thm414}, there is some maximum $z \in \clos{0}{1}$ such that $\G \cup \neg \D$ is $\clos{z}{1}$-consistent.
  Then $z < x$, since if $z \geq x$ the argument $\sq{\G}{\D}$ would not be $\clos{x}{1}$-symmetric valid; so choose some $y$ such that $z < y < x$ and let $\alpha = \clos{y}{1}$, noting that $\alpha \supsetneq \clos{x}{1}$.
  By $z$'s maximality, we know that $\Sigma$ is $\alpha$-inconsistent; but then $\sq{\G}{\D}$ is $\alpha$-symmetric valid.
\end{proof}

If we imagine starting a process at the widest upset $\hopen{0}{1}$ and gradually narrowing all the way to the other extreme $\{1\}$, we now have a clear picture of the result.
While arguments certainly move from invalid to valid as we narrow our upset, this doesn't happen just anywhere.
First, \Cref{fact:sym-distinct-iff-rational} ensures us that this never happens in the move from $\clos{x}{1}$ to $\hopen{x}{1}$ where $x$ is irrational.
And \Cref{fact:limit} ensures that this never happens as we move to a \emph{closed} upset at all, regardless of whether its threshold is rational or irrational---any argument validated by the symmetric consequence at a closed upset was \emph{already} validated by some properly wider upset.
The only time a new argument can become valid, then, is in the shift from $\clos{x}{1}$ to $\hopen{x}{1}$ when $x$ is rational---and \Cref{fact:sym-distinct-iff-rational} ensures that this indeed happens for every rational $x$.

Moreover, this process indeed reaches classical logic right at $\{1\}$, just as our natural picture requires.
(By \Cref{fact:sym-closed-distinct}, it cannot reach classical logic any earlier.)
We can see that the way this happens is very different from the way that preservation consequence reaches supervaluational logic at $\{1\}$.
As we saw above, there is a very large class of arguments, including most classically valid \setfmla\ arguments of any interest, that are $\{1\}$-preservation valid but $\alpha$-preservation invalid for all other $\alpha$.
Preservation consequence has a massive leap exactly at $\{1\}$ in its \setfmla\ fragment, and a corresponding massive leap exactly at $\hopen{0}{1}$ in its \fmlaset\ fragment.
Symmetric consequence, by contrast, has no sudden leap at $\{1\}$; by \Cref{fact:limit} it adds no valid arguments that were not already counted as valid on some wider upset, and this holds even for the full \setset\ framework.

\subsection{Argument size}
\label{sec:argument-size}

So far, this has all left open any questions about \emph{which} arguments become valid \emph{when} as the upset for symmetric consequence narrows.
But in fact we can describe this process more precisely, pinpointing for at least some arguments exactly where in this process they move from invalid to valid, and giving for every argument an upper bound.

To do this, we'll use the notion of the \demph{size} of an argument and another result from \cite{knight2002measuring}:

\begin{defn}
  The \demph{size} of the argument $\sq{\G}{\D}$ is $|\G \cup \neg \D|$.\footnote{In fact, the results to follow would also hold if size counted members of $\G \cup \neg \D$ only up to classical equivalence; and this notion might be better-behaved in some ways.
  But we kept it simple.}
  An argument $\sq{\G'}{\D'}$ is a \demph{subargument} of the argument $\sq{\G}{\D}$ iff $\G' \subseteq \G$ and $\D' \subseteq \D$, and a \demph{proper} subargument iff it is a subargument and they are distinct.
\end{defn}

\begin{fact}[\protect{\citealt[Thm.\ 3.5, p.\ 80]{knight2002measuring}}] \label{fact:sym-min-size}
  If $\sq{\G}{\D}$ is classically valid, but no proper subargument of it is classically valid, then where $n$ is its size, it is $\clos{\frac{n - 1}{n}}{1}$-symmetric invalid and $\hopen{\frac{n - 1}{n}}{1}$-symmetric valid.
\end{fact}

This immediately settles the situation for very many arguments.
For example, recall from \Cref{sec:disc-pres} the $n$-ary conjunction introduction arguments, where $CI_{n}$ is $\sq{p_{1}, \ldots, p_{n}}{\bigwedge p_{i}}$.
Note that these all meet the conditions of \Cref{fact:sym-min-size}, and that $CI_{n}$ has size $n + 1$.
So for any $n$, the argument $CI_{n}$ is $\clos{\frac{n}{n + 1}}{1}$-invalid and $\hopen{\frac{n}{n+1}}{1}$-valid.
As $\alpha$ narrows, we can $\alpha$-symmetric validly conjoin larger and larger collections of conjuncts.

Or let $MP_{n}$ be the argument $\sq{p_{1}, p_{1} \hook p_{2}, \ldots, p_{n - 1} \hook p_{n}}{p_{n}}$.
Again, all of these meet the conditions of \Cref{fact:sym-min-size}, and $MP_{n}$ has size $n + 1$.
So for any $\alpha$, $MP_{n}$ is $\alpha$-symmetric valid iff $CI_{n}$ is---again, as $\alpha$ narrows, we can $\alpha$-symmetric validly detach longer and longer chains of material conditionals.

\Cref{fact:sym-min-size} gives us the exact place where certain arguments move from invalid to valid as we narrow our upset, but it only covers arguments with no classically-valid proper subarguments.
We can use it, however, to get an upper bound that applies to all arguments:\footnote{\cite{knight2002measuring} also gives interesting and useful results for cases where $\G \cup \neg \D$ is inconsistent but not minimally so, achieving tighter upper bounds than this; but there are complexities involved that we prefer to avoid here.}

\begin{fact} \label{fact:sym-size}
  If $\sq{\G}{\D}$ is classically valid and has size $n$, then it is $\hopen{\frac{n - 1}{n}}{1}$-symmetric valid.
\end{fact}

\begin{proof}
  If $\sq{\G}{\D}$ is classically valid, then it contains some subargument $\sq{\G'}{\D'}$ that is minimally classically valid; let this subargument have size $m$.
  By \Cref{fact:sym-min-size}, this subargument is $\hopen{\frac{m - 1}{m}}{1}$-symmetric valid; and so by \Cref{fact:sym-monotonic} $\sq{\G}{\D}$ is as well.
  But since $\sq{\G'}{\D'}$ is a subargument of $\sq{\G}{\D}$, we know that $m \leq n$, and so $\hopen{\frac{n - 1}{n}}{1} \subseteq \hopen{\frac{m - 1}{m}}{1}$.
  By \Cref{fact:sym-order}, then, we have our result.
\end{proof}

This fleshes out our natural thought, giving us more detail about when various classically-valid arguments settle into $\alpha$-symmetric validity.
It also lets us determine, for any classically valid argument, some degree of probability that is not perfect certainty but still high enough to ensure that the argument's premises can't all have probability that high while all its conclusions have symmetrically low probabilities.

\subsection{Tarskian properties, weak paraconsistency, weak paracompleteness}
\label{sec:sym-tarskian}

We now turn to questions of the Tarskian-ness (or otherwise) of symmetric consequence relations, and also consider their weak paraconsistency and weak paracompleteness.

We have already seen in \Cref{fact:sym-monotonic} that all $\alpha$-symmetric consequence relations are monotonic.
But since symmetric consequence is not defined by preserving any single status across all models, it is worth asking about reflexivity and transitivity as well.\footnote{There is more on this topic in \cite{humberstone:hl,fr:valuations}.}

\begin{fact} \label{fact:sym-reflexive}
  $\alpha$-symmetric consequence is reflexive iff $.5 \not\in \alpha$.
\end{fact}

\begin{proof}
  If $.5 \in \alpha$, then also $.5 \in \malpha$.
  In this case, take a two-world model $\tuple{\{u, v\}, \wp(\{u, v\}), \dn{\;}, \prob}$ such that $\dn{p} = \{u\}$ and $\prob(\{u\}) = \prob(\{v\}) = .5$; this is an $\alpha$-symmetric counterexample to $\sq{p}{p}$.

  On the other hand, suppose we have any $\alpha$-symmetric counterexample $\tuple{W, \alg, \dn{\;}, \prob}$ to $\sq{p}{p}$.
  Then we must have $\prob(p) \in \alpha \cap \malpha$; but if $\alpha \cap \malpha$ is nonempty then $.5 \in \alpha$.
\end{proof}

\begin{fact} \label{fact:sym-transitive}
  $\alpha$-symmetric consequence is transitive iff $.5 \in \alpha$ or $\alpha = \{1\}$.
\end{fact}

\begin{proof}
  By \Cref{fact:sym-1-classical} (and the fact that classical consequence is transitive), we know that $\{1\}$-symmetric consequence is transitive.
  So let $\alpha \neq \{1\}$.

  First, we show that if $.5 \in \alpha$ then $\alpha$-symmetric consequence is transitive.
  To that end, we show the contrapositive: that if $\sq{\G}{\D}$ is $\alpha$-symmetric invalid, then at least one of the arguments $\sq{\G}{\D, \phi}$ or $\sq{\phi, \G}{\D}$ is $\alpha$-symmetric invalid as well.
  To see this, let $.5 \in \alpha$, and take an $\alpha$-symmetric counterexample $\tuple{W, \alg, \dn{\;}, \prob}$ to $\sq{\G}{\D}$.
  If $\prob(\phi) \geq .5$, then this is also an $\alpha$-symmetric counterexample to $\sq{\phi, \G}{\D}$; and if $\prob(\phi) \leq .5$, then this is also an $\alpha$-symmetric counterexample to $\sq{\G}{\D, \phi}$.
  So one or the other of those arguments indeed must be $\alpha$-symmetric invalid.

  Now, we show that if $.5 \not\in \alpha \neq \{1\}$, then $\alpha$-symmetric consequence is not transitive.
  Return to our arguments $CI_{n}$ discussed above.
  By our assumptions on $\alpha$, we know from \Cref{fact:sym-min-size} that $CI_{2}$ is $\alpha$-symmetric valid and that there is some $k$ such that $CI_{k}$ is not $\alpha$-symmetric valid.
  But it is quick to see that any monotonic and transitive consequence relation where $CI_{2}$ is valid must validate $CI_{n}$ for all $n$.
  Since we know from \Cref{fact:sym-monotonic} that $\alpha$-symmetric consequence is monotonic, it must not be transitive.
\end{proof}

\Cref{fact:sym-reflexive,fact:sym-monotonic,fact:sym-transitive} together show that, among $\alpha$-symmetric consequence relations, only one is fully Tarskian: the limit case of $\{1\}$-symmetric consequence, which is classical consequence.
For all other choices of $\alpha$, either $.5 \in \alpha$ and we have a nonreflexive (but transitive) consequence relation, or $.5 \not\in \alpha$ and we have a nontransitive (but reflexive) relation.

We also have some easy facts about weak paraconsistency and weak paracompleteness, that again connect to the key question whether $.5 \in \alpha$:
\begin{fact} \label{fact:sym-weak-para-nonreflexive}
  For any $\alpha$, the following are equivalent:
  \begin{itemize}
    \item $\alpha$-symmetric consequence is weakly paraconsistent;
    \item $\alpha$-symmetric consequence is weakly paracomplete;
    \item $\alpha$-symmetric consequence is nonreflexive.
  \end{itemize}
\end{fact}

\begin{proof}
 Weak paraconsistency implies weak paracompleteness: let $\sq{\phi, \neg \phi}{\D}$ be $\alpha$-symmetric invalid.
 Then by self-duality, $\sq{\neg \D}{\neg \phi, \neg \neg \phi}$ is invalid; this shows weak paracompleteness.

 Weak paracompleteness implies nonreflexivity: let $\sq{\G}{\phi, \neg \phi}$ be $\alpha$-symmetric invalid.
 Then by \Cref{fact:sym-negation} $\sq{\G, \phi}{\phi}$ is $\alpha$-symmetric invalid, and so by \Cref{fact:sym-monotonic} $\sq{\phi}{\phi}$ must be $\alpha$-symmetric invalid.

 Nonreflexivity implies weak paraconsistency: let $\sq{\phi}{\phi}$ be $\alpha$-symmetric invalid.
 Then by \Cref{fact:sym-negation} $\sq{\phi, \neg \phi}{\emptyset}$ is $\alpha$-symmetric invalid.
 Since $\phi \land \neg \phi$ can never take probability other than 0, any counterexample to $\sq{\phi, \neg \phi}{\emptyset}$ is also a counterexample to $\sq{\phi, \neg \phi}{\phi \land \neg \phi}$.
\end{proof}

\subsection{Relations to preservation consequence}
\label{sec:relat-pres-cons}

We close our discussion of symmetric consequence relations by pointing out some relations of relative strength that connect symmetric and preservation consequence relations.
Recall that from \Cref{fact:sym-order} we know that if $\alpha \subseteq \dalpha$ then $\alpha$-symmetric consequence is at least as strong as $\dalpha$-symmetric consequence, and if $\dalpha \subseteq \alpha$ then $\dalpha$-symmetric consequence is at least as strong as $\alpha$-symmetric consequence.
In fact, we can show that either way, $\alpha$-preservation consequence always lies somewhere properly in between:

\begin{fact} \label{fact:sym-pres-between-dual}
  For any $\alpha$, the consequence relation of $\alpha$-preservation is intermediate in strength between $\alpha$-symmetric consequence and $\dalpha$-symmetric consequence.
\end{fact}

\begin{proof}
  Without loss of generality, suppose that $\alpha \subseteq \dalpha$,\footnote{If instead $\dalpha \subseteq \alpha$, then let $\beta = \dalpha$ and apply the result to $\beta$.} and so by \Cref{fact:sym-order} if any argument $\sq{\G}{\D}$ is $\dalpha$-symmetric valid it must also be $\alpha$-symmetric valid.

  We show first that if any $\sq{\G}{\D}$ is $\dalpha$-symmetric valid then it is $\alpha$-preservation valid.
  Suppose that $\modl = \tuple{W, \alg, \dn{\;}, \prob}$ is an $\alpha$-preservation counterexample to $\sq{\G}{\D}$; then $\prob[\G] \subseteq \alpha$ and $\prob[\D] \subseteq \clos{0}{1} \setminus \alpha$.
  Since $\alpha \subseteq \dalpha$, we have that $\prob[\G] \subseteq \dalpha$, and $\clos{0}{1} \setminus \alpha = \mirror{\dalpha}$, so $\modl$ is in fact a $\dalpha$-symmetric counterexample to $\sq{\G}{\D}$ as well.

  Next we show that if any $\sq{\G}{\D}$ is $\alpha$-preservation valid then it is $\alpha$-symmetric valid.
  Suppose that $\modl = \tuple{W, \alg, \dn{\;}, \prob}$ is an $\alpha$-symmetric counterexample to $\sq{\G}{\D}$; then $\prob[\G] \subseteq \alpha$ and $\prob[\D] \subseteq \malpha = \clos{0}{1} \setminus \dalpha$.
  But since $\alpha \subseteq \dalpha$, we know that $\clos{0}{1} \setminus \dalpha \subseteq \clos{0}{1} \setminus \alpha$, so $\prob[\D] \subseteq \clos{0}{1} \setminus \alpha$, and $\modl$ is in fact an $\alpha$-preservation counterexample to $\sq{\G}{\D}$ as well.
\end{proof}

Moreover, we can also show that preservation consequence and symmetric consequence never quite meet:

\begin{fact}
  For any upsets $\alpha, \beta$, we have that $\alpha$-symmetric consequence is distinct from $\beta$-preservation consequence.
\end{fact}

\begin{proof}
  By \Cref{fact:sym-weak-para-nonreflexive}, $\alpha$-symmetric consequence is either: both weakly paracomplete and weakly paraconsistent, or neither weakly paracomplete nor weakly paraconsistent.
  By \Cref{fact:weak-para-exactly-one}, $\beta$-preservation consequence is either: weakly paracomplete and not weakly paraconsistent, or weakly paraconsistent and not weakly paracomplete.
  So they cannot be the same.
\end{proof}

\section{Conclusion}
\label{sec:conclusion}

In this paper we have explored three notions of probabilistic consequence, which we called material consequence, preservation consequence, and symmetric consequence. Preservation consequence is the most central of these three, since material consequence and symmetric consequence can be both be reduced to it. Indeed, to say that $\Gamma$ $\alpha$ materially entails $\Delta$ is to say that $\bigwedge \Gamma \hook \bigvee \Delta$ is $\alpha$-preservation valid. And to say that $\Gamma$ $\alpha$-symmetrically entails $\Delta$ is to say that $\Gamma \cup \neg \Delta$ $\alpha$-entails a contradiction.

All three relations of consequence coincide with classical logic in specific cases, but while material consequence coincides with classical logic for any $\alpha$,  symmetric consequence coincides with classical logic only in the case of $\alpha=\{1\}$; preservation consequence too coincides with classical logic for $\alpha=\{1\}$, but only in the \setfmla\ setting. In the \setset\ setting, as we have seen, certainty preservation yields supervaluationism rather than classical logic. Figure \ref{fig:containments} gives a representation of the containment relations between these logics.
In this figure, a logic is included in another when it is lower and connected by a solid edge; thick lines indicate continuum many logics, either contained one in another in the case of solid thick lines or all incomparable in the case of dotted thick lines.

Our exploration in this paper leaves us with some open questions. In particular, we have yet to prove or to disprove \Cref{conj:setset-captured} above, and stating that nontrivial \setset\ cases of $\alpha$-entailment supervene on classical entailment between a specific premise and a specific conclusion. On a more philosophical level, we have yet to examine how these various consequence relations can help us handle specific arguments, such as the lottery paradox, the preface paradox, and the sorites paradox --- invoked in particular by \cite{adams1998primer} and \cite{knight2002measuring}, and which also motivated the fuzzy counterpart of symmetric consequence relations found in \cite{smith2008vagueness,cobreros2024tolerance}. We leave this investigation, as well as the issue of which consequence relation or relations to favor and to build upon, for further work.

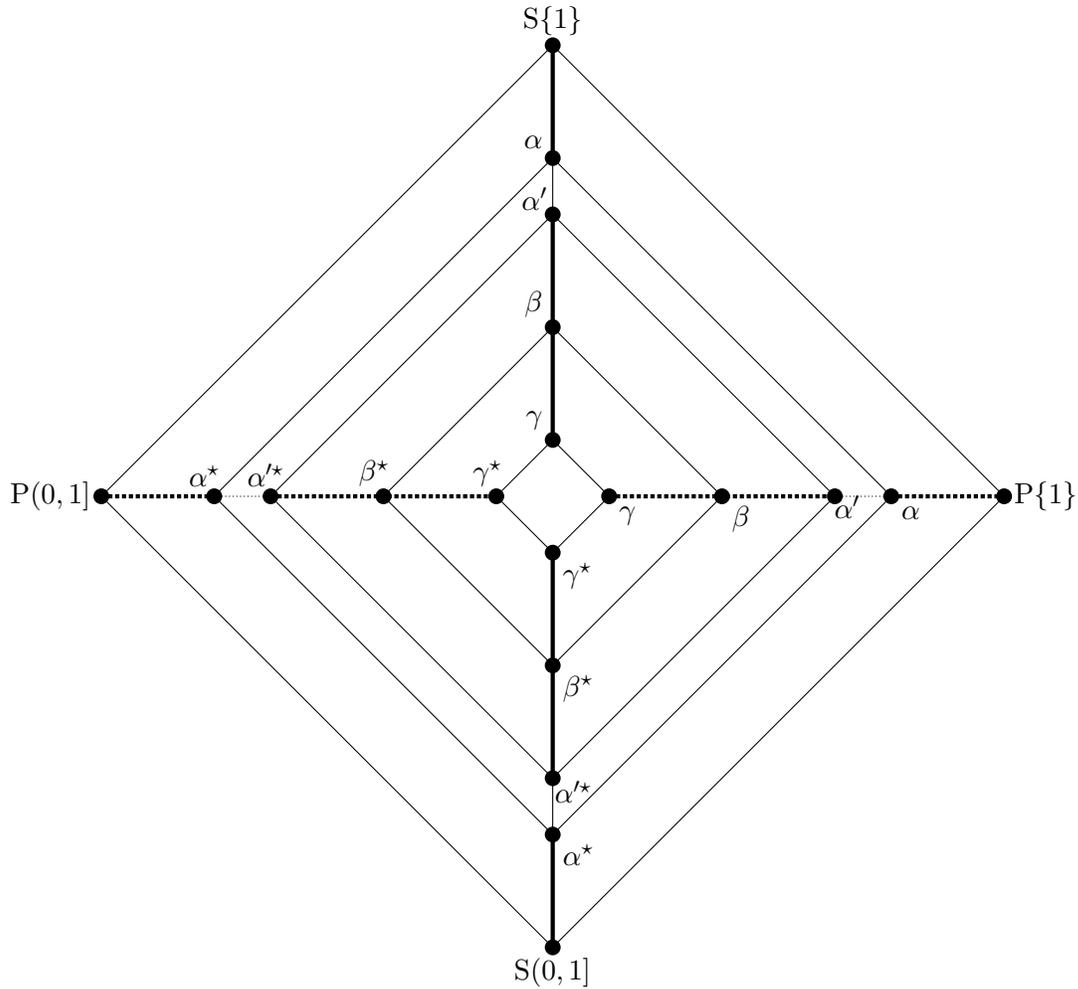
\begin{figure}
  {\centering
  \begin{tikzpicture}[scale=1.5]
    \coordinate [label=below:S$\hopen{0}{1}$]  (almostempty) at (5, 1);
    \coordinate [label=315:\dual{\alpha}] (dbetap) at (5, 2);
    \coordinate [label={[label distance=-1.5mm]315:\dual{\alpha'}}] (dbeta) at (5, 2.5);
    \coordinate [label=315:\dual{\beta}] (dgamma) at (5, 3.5);
    \coordinate [label=315:\dual{\gamma}]  (halfclosed)  at (5, 4.5);
    \coordinate [label=135:$\gamma$] (halfopen)    at (5, 5.5);
    \coordinate [label=135:$\beta$] (gamma)    at (5, 6.5);
    \coordinate [label={[label distance=-1mm]135:$\alpha'$}]   (beta)  at (5, 7.5);
    \coordinate [label=135:$\alpha$] (betap)  at (5, 8);
    \coordinate [label=S$\{1\}$]         (classical)   at (5, 9);

    \coordinate [label=left:P$\hopen{0}{1}$]  (widepres)    at (1, 5);
    \coordinate [label=135:$\dual{\alpha}\mkern-10mu$] (pdbetap) at (2, 5);
    \coordinate [label=135:$\dual{\alpha'}\mkern-15mu$] (pdbeta) at (2.5, 5);
    \coordinate [label=135:$\dual{\beta}\mkern-10mu$] (pdgamma) at (3.5, 5);
    \coordinate [label=135:$\dual{\gamma}\mkern-10mu$] (phalfclosed) at (4.5, 5);
    \coordinate [label=315:$\gamma$]  (phalfopen) at (5.5, 5);
    \coordinate [label=315:$\beta$]  (pgamma) at (6.5, 5);
    \coordinate [label={[label distance=-2mm]315:$\alpha'\mkern-10mu$}]  (pbeta) at (7.5, 5);
    \coordinate [label=315:$\alpha$]  (pbetap) at (8, 5);
    \coordinate [label=right:P$\{1\}$]         (tightpres)   at (9, 5);

    \foreach \point in {almostempty, halfclosed, halfopen, classical, widepres, tightpres, phalfclosed, phalfopen, beta, dbeta, pbeta, pdbeta, dbetap, pbetap, betap, pdbetap, dgamma, gamma, pgamma, pdgamma}
      \fill [black] (\point) circle (2pt);

    \draw[ultra thick] (almostempty) -- (dbetap);
    \draw (dbetap) -- (dbeta);
    \draw[ultra thick] (dbeta) -- (halfclosed);
    \draw[ultra thick] (halfopen) -- (beta);
    \draw (beta) -- (betap);
    \draw[ultra thick] (betap) -- (classical);
    \draw[densely dotted, ultra thick] (widepres) -- (pdbetap);
    \draw[densely dotted] (pdbetap) -- (pdbeta);
    \draw[densely dotted, ultra thick] (pdbeta) -- (phalfclosed);
    \draw[densely dotted, ultra thick] (phalfopen) -- (pbeta);
    \draw[densely dotted] (pbeta) -- (pbetap);
    \draw[densely dotted, ultra thick] (pbetap) -- (tightpres);

    \draw (halfclosed) -- (phalfclosed);
    \draw (phalfclosed) -- (halfopen);
    \draw (halfclosed) -- (phalfopen);
    \draw (phalfopen) -- (halfopen);

    \draw (almostempty) -- (widepres);
    \draw (almostempty) -- (tightpres);
    \draw (widepres) -- (classical);
    \draw (tightpres) -- (classical);

    \draw (dbeta) -- (pbeta);
    \draw (dbeta) -- (pdbeta);
    \draw (pbeta) -- (beta);
    \draw (pdbeta) -- (beta);
    \draw (dbetap) -- (pbetap);
    \draw (dbetap) -- (pdbetap);
    \draw (pbetap) -- (betap);
    \draw (pdbetap) -- (betap);

    \draw (dgamma) -- (pgamma);
    \draw (dgamma) -- (pdgamma);
    \draw (pgamma) -- (gamma);
    \draw (pdgamma) -- (gamma);
  \end{tikzpicture}
  }

  \noindent Symmetric consequence relations pictured on the vertical axis; preservation on the horizontal.
  Solid lines indicate containment; dotted lines indicate incomparability.
  Thick lines indicate continuum many consequence relations.
  $\gamma$ is the upset $\hopen{.5}{1}$.
  Zooming in at any rational $z > .5$ reveals the indicated structure where $\alpha$ is $\hopen{z}{1}$ and $\alpha'$ is $\clos{z}{1}$; zooming in at any irrational $x > .5$ reveals the indicated structure where $\beta$ is both $\hopen{x}{1}$ and $\clos{x}{1}$.

  P$\hopen{0}{1}$ is also subvaluational, P$\{1\}$ supervaluational, and S$\{1\}$ is classical logic as well as every material consequence relation.

  \caption{Containments among consequence relations}
  \label{fig:containments}
\end{figure}

%Reducing it all to preservation
%Comparisons between three kinds of consequence
%
%Conclusion, sorites applications, lottery applications, conditional probability, etc

\newpage

\section*{Acknowledgments}

We thank two anonymous referees for their helpful comments, as well as several audiences, in particular at the Melbourne Logic Workshop held at Monash University in 2023, at the University of Queensland, and at the workshop Conditionals 2024 held in Barcelona. We also thank several colleagues for helpful exchanges and discussions on topics connected to this paper, including Guillermo Badia, Quentin Blomet, Pablo Cobreros, Alba Cuenca, Tommaso Flaminio, Peter Fritz, Lluis Godo, Gabriele Kern-Isberner, Serafina Lapenta, Lorenzo Rossi, Hans Rott, Giuseppe Sanfilippo, Jan Sprenger, and Benjamin Spector. This research received support from PLEXUS (Grant Agreement no 101086295), a Marie Sklodowska-Curie action funded by the EU under the Horizon Europe Research and Innovation Programme. We also thank the ANR programs PROBASEM (ANR-19-CE28-0004) and FRONTCOG (17-EURE-0017). DR thanks the ARC project ``Substructural logics for limited resources''. PE thanks Institut Jean-Nicod (CNRS, ENS-PSL, EHESS), and Monash University for hosting him during the writing of this paper.

\bibliographystyle{apalike}
\bibliography{../prob}
\end{document}